\documentclass[a4paper,fleqn, 10pt]{article}
 \usepackage{hyperref,amsmath,makeidx,amsfonts,a4wide,amssymb,amsthm,longtable, color}
\newtheorem{theorem}{Theorem}

\numberwithin{equation}{section}
\numberwithin{lemma}{section}
\numberwithin{theorem}{section}
\numberwithin{corollary}{section}

\allowdisplaybreaks

\begin{document}
\title{On the discrete analogues of Appell function $F_3$}
\author{Ravi Dwivedi$^{1,}$\footnote{E-mail: dwivedir999@gmail.com}   \, and Vivek Sahai$^{2,}$\footnote{E-mail: sahai\_vivek@hotmail.com (Corresponding author)} \\ ${}^{1}$Department of Science, SAGEMMC, Jagdalpur, Bastar, CG, 494001, India; \\ ${}^{2}$Department of Mathematics and Astronomy, Lucknow University, \\ Lucknow 226007, India.}
	\maketitle 
\begin{abstract}
The present paper is devoted to the study of Appell hypergeometric function $F_3$ from discrete point of view. We mainly introduce two generalized discrete forms of $F_3$ and study their basic properties \emph{viz.} regions of convergence, difference-differential equations, integral representations, finite and infinite sums, recursion formulae, to name a few. 


\end{abstract}
	\section{Introduction}
	The Appell function $F_3$ and its two degenerations (functions obtained by taking limit) $\varXi_1$ and $\varXi_2$ have been introduced by Appell and Humbert, written fairly in \cite{emo, kdf, sk}, and defined by 
	\begin{align}
		F_3 (a, a', b, b'; c; x, y) & = \sum_{m, n = 0}^{\infty} \frac{(a)_{m} \, (a')_n \, (b)_m  \, (b')_n}{(c)_{m + n}} \, \frac{x^m \, y^n}{m ! \, n!}, \quad \vert x \vert, \, \vert y \vert < 1;
	\end{align}
\begin{align}
	  \varXi_1 \left(a_1, a_2, b_1; c;  x,  y\right)
	  & = \sum_{m,n \geq 0} \frac{(a_1)_{m} \, (a_2)_n \, (b_1)_m}{ (c)_{m+n} \, m! \, n!} \ x^m \, y^n;
\end{align}
and 
\begin{align}
 \varXi_2 \left(a_1, b_1; c;  x,  y\right)
 & = \sum_{m,n \geq 0} \frac{(a_1)_{m} \, (b_1)_m}{(c)_{m+n} \, m! \, n!} \ x^m \, y^n,
\end{align}	
	 respectively. So far, the appearance of these functions in various physical problems increased the interest of mathematicians and motivated them to study these functions from different aspects. 
	  In this paper, we explore these functions from discrete point of view. In particular, we give two discrete forms of Appell function $F_3$ and obtain their regions of existence, differential as well as integral properties. 

This paper is third in a series, wherein discrete Appell functions are being examined. The discrete forms of Appell function $F_1$ \cite{ds16} and Appell function $F_2$ \cite{ds17} have been thoroughly studied. Let  $a_1$, $a_2$, $b_1$, $b_2$, $c$, $t$, $t_1$ and $t_2$ be complex number such that $\Re (c) \ne 0, -1, \dots$. Then, for $k$, $k_1$ $k_2 \in \mathbb{N}$, we define the discrete forms of Appell function  ${F}_3$ as follows: 
	\begin{align}
		\mathcal{F}^{(1)}_3 & = \mathcal{F}^{(1)}_3(a_1, a_2, b_1, b_2; c; t_1, t_2, k_1, k_2, x, y)\nonumber\\ 
		& = \sum_{m,n\geq0} \frac{(a_1)_{m} \, (a_2)_n \, (b_1)_m \, (b_2)_n \, (-1)^{m k_1} \, (-t_1)_{mk_1} \, (-1)^{n k_2} \, (-t_2)_{nk_2}}{ (c)_{m+n} \, m! \, n!} \ x^m \, y^n; \label{3.1}
	\end{align} 
	\begin{align}
		\mathcal{F}^{(2)}_3 & =	\mathcal{F}^{(2)}_3(a_1, a_2, b_1, b_2; c; t, k, x, y)\nonumber\\
		& = \sum_{m,n\geq0} \frac{(a_1)_{m} \, (a_2)_n \, (b_1)_m \, (b_2)_n \, (-1)^{(m + n) k} \, (-t)_{(m + n)k}}{(c)_{m+n} \, m! \, n!} \ x^m \, y^n; \label{3.2}
	\end{align} 
where $(-t)_{m \, k}$ is defined by \cite{bc4}
\begin{align}
	(-t)_{m \, k} &  = 	k^{m \, k} \, \prod_{i = 0}^{k - 1} \left(\frac{-t + i}{k}\right)_m.
\end{align}
In the next four sections, we explore the discrete Appell function $\mathcal{F}_3^{(1)}$ and later in the Section~6, we list the results for discrete function $\mathcal{F}_3^{(2)}$. 

\section{The discrete Appell function $\mathcal{F}_3^{(1)}$} 
We can see that the special values of $k_1$ and $k_2$ lead $\mathcal{F}_3^{(1)}$ into well known classical functions. In particular, for $k_1 = k_2 = 0$, we have
\begin{align}
	 \mathcal{F}^{(1)}_3(a_1, a_2, b_1, b_2; c; t_1, t_2, 0, 0, x, y)
	 = F_3(a_1, a_2, b_1, b_2; c; x, y).
\end{align} 
For $k_1 = 1$, $k_2 = 0$, we get a Kamp\'e de F\'eriet  function 
 \begin{align}
	& \mathcal{F}^{(1)}_3(a_1, a_2, b_1, b_2; c; t_1, t_2, 1, 0, x, y)
	 = F_{{1}:{0}, {0}} ^{{0}:{3}, {2}}\left(\begin{array}{ccc}
		-: & a_1, b_1, -t_1, & a_2, b_2\\
		c: & - & - 
	\end{array}; -x, \  y\right), 
\end{align} 
 where $ F_{{1}:{0}, {0}} ^{{0}:{3}, {2}}$ is the Kamp\'e de F\'eriet  hypergeometric function;  defined in general by \cite{sk}
  \begin{align}
  	& F_{{l_2}:{l'_2}, {l''_2}} ^{{l_1}:{l'_1}, {l''_1}}\left(\begin{array}{ccc}
  			A: & B, & C\\
  			D: &E, &F 
  		\end{array}; x, \ y\right)\nonumber\\
  	& = 
  	\sum_{m,n\geq 0} \ \frac{\prod_{i=1}^{l_1} (a_i)_{m+n} \, \prod_{i=1}^{l'_1} (b_i)_{m} \, \prod_{i=1}^{l''_1} (c_i)_{n}}{\prod_{i=1}^{l_2} (d_i)_{m+n}   \, \prod_{i=1}^{l'_2} (e_i)_{m}  \, \prod_{i=1}^{l''_2} (f_i)_{n}} \ \frac{x^m \, y^n}{m! \, n!},\label{c1eq71}
  \end{align}
 where $A$ denote the sequence of complex numbers $a_1, \dots, a_{l_1}$.   Further, for $k_1 = 0$, $k_2 = 1$ and $k_1 = k_2 = 1$, we have the following functions
\begin{align}
& \mathcal{F}^{(1)}_3(a_1, a_2, b_1, b_2; c; t_1, t_2, 0, 1, x, y)
 = F_{{1}:{0}, {0}} ^{{0}:{2}, {3}}\left(\begin{array}{ccc}
	-: & a_1, b_1,  & a_2, b_2, -t_2\\
	c: & - & - 
\end{array}; x, \  -y\right),
\end{align}
and
\begin{align}
	 \mathcal{F}^{(1)}_3(a_1, a_2, b_1, b_2; c; t_1, t_2, 1, 1, x, y) 
	 = F_{{1}:{0}, {0}} ^{{0}:{3}, {3}}\left(\begin{array}{ccc}
		-: & a_1, b_1, -t_1, & a_2, b_2, -t_2\\
		c: & - & - 
	\end{array}; -x, \  -y\right),
\end{align}
	respectively. 
	
For convergence of the discrete function $\mathcal{F}^{(1)}_3$, we consider its general term $\mathcal{A}_{m, n} x^m \, y^n$. We have
  \begin{align}
  &	\left \vert \mathcal{A}_{m, n} x^m \, y^n \right \vert\nonumber\\
   & = \left \vert \frac{(a)_{m+n} \, (b_1)_m \, (b_2)_n \, (-t_1)_{mk_1} \, (-t_2)_{nk_2}}{ (c)_{m+n} \, m! \, n!} \ x^m \, y^n \right \vert\nonumber\\
  	& < \left \vert \frac{\Gamma (c)}{\Gamma (a) \, \Gamma (b_1) \, \Gamma (b_2) \, \Gamma (-t_1) \, \Gamma (-t_2)} \right \vert \nonumber\\
  	& \quad \times \left \vert \frac{\Gamma (a + m + n) \, \Gamma (b_1 + m) \, \Gamma (b_2 + n) \, \Gamma (-t_1 + mk_1) \, \Gamma (-t_2 + nk_2)}{\Gamma (c + m + n) \, \Gamma (m + 1) \, \Gamma (n + 1)} \right \vert \, \vert x \vert^m \, \vert y\vert^n.
  \end{align}
The Stirling formula $\lim_{n \to \infty} \Gamma (\lambda + n) = \sqrt (2 \pi) \, n^{\lambda + n - \frac{1}{2}} \, e^{-n}$, for large values of $m$ and $n$ yields
\begin{align}
	&	\left \vert \mathcal{A}_{m, n} x^m \, y^n \right \vert\nonumber\\
	& < \left \vert \frac{\Gamma (c)}{\Gamma (a) \, \Gamma (b_1) \, \Gamma (b_2) \, \Gamma (-t_1) \, \Gamma (-t_2)} \right \vert \nonumber\\
	& \quad \times \left \vert 2 \pi \, (m + n)^{a - c} \, m^{b_1 - 1} \, n^{b_2 - 1} \, (mk_1)^{mk_1 - t_1 - \frac{1}{2}} \, (nk_2)^{nk_2 - t_2 - \frac{1}{2}} e^{- (m k_1 + n k_2)} \right \vert \, \vert x \vert^m \, \vert y\vert^n.
\end{align}
Consider $N > \left \vert \frac{2 \pi \, \Gamma (c)}{\Gamma (a) \, \Gamma (b_1) \, \Gamma (b_2) \, \Gamma (-t_1) \, \Gamma (-t_2)} \right \vert$. Then
\begin{align}
	&	\left \vert \mathcal{A}_{m, n} x^m \, y^n \right \vert < \frac{N \, (mk_1)^{mk_1 - t_1 - \frac{1}{2}} \, (nk_2)^{nk_2 - t_2 - \frac{1}{2}}}{ (m + n)^{c - a} \, m^{1 - b_1} \, n^{1 - b_2} \, e^{m k_1 + n k_2}} \, \vert x \vert^m \, \vert y\vert^n.
\end{align}
For $k \in \mathbb{N}, t_1, t_2 \in \mathbb{C}$ and $\vert x\vert < 1, \vert y \vert < 1$, $\left \vert \mathcal{A}_{m, n} x^m \, y^n \right \vert \to 0$ as $m, n \to \infty$. Hence the discrete function $\mathcal{F}^{(1)}_3$ converges absolutely. 

\subsection{Difference equations}

Let $\Theta_t := t \, \rho_t \, \Delta_t$, where $\Delta_t f(t) = f(t + 1) - f(t)$ and $\rho_t f (t) = f(t - 1)$, be the difference operator. It can be readily verified that
\begin{align}
	\Theta_t \, ((-1)^{nk} \, (-t)_{nk}) & 
	 = n\, k \, (-1)^{nk} \, (-t)_{nk}. 
\end{align} 
We have the following theorem:
\begin{theorem} 
	The discrete Appell function $\mathcal{F}^{(1)}_3$ satisfy the following difference equations:
\begin{align}
	\left[\Theta_{t_1} \left(\frac{1}{k_1} \Theta_{t_1} + \frac{1}{k_2} \Theta_{t_2} + c - 1\right) - k_1 \, (-1)^{k_1} \, (-t_1)_{k_1}\, x \, \rho_{t_1}^{k_1}  \left(\frac{1}{k_1} \Theta_{t_1} + a_1\right) \left(\frac{1}{k_1} \Theta_{t_1}  + b_1\right) \right] \mathcal{F}^{(1)}_3 = 0;\label{1.15}
\end{align}	    
\begin{align}
	\left[\Theta_{t_2}  \left(\frac{1}{k_1} \Theta_{t_1} + \frac{1}{k_2} \Theta_{t_2} + c - 1\right) - k_2 \, (-1)^{k_2}  (-t_2)_{k_2} \, y \, \rho_{t_2}^{k_2}  \left(\frac{1}{k_2} \Theta_{t_2} + a_2\right) \left(\frac{1}{k_2} \Theta_{t_2}  + b_2 \right) \right] \mathcal{F}^{(1)}_3 = 0.\label{1.16}
\end{align}
\end{theorem}
\begin{proof}
To prove equation \eqref{1.15}, we start with 
\begin{align}
	& \Theta_{t_1} \left(\frac{1}{k_1} \Theta_{t_1} +\frac{1}{k_2} \Theta_{t_2} + c - 1\right) \, \mathcal{F}^{(1)}_3\nonumber \\
	& = \sum_{m,n \geq 0} \frac{(a_1)_{m} \, (a_2)_n \, (b_1)_m \, (b_2)_n \, (-1)^{m k_1} \, (-t_1)_{mk_1} \, (-1)^{n k_2} \, (-t_2)_{nk_2}}{ (c)_{m+n} \, m! \, n!} \ x^m \, y^n \, mk_1 \, (c + m + n - 1)\nonumber\\
	& = k_1 \, \sum_{m, n \ge 0} (a_1 + m) \, (b_1 + m) \, (-1)^{k_1} \, (-t_1)_{k_1} \, x \, \rho_{t_1}^{k_1} \nonumber\\
	& \quad \times \frac{(a_1)_{m} \, (a_2)_n \, (b_1)_{m} \, (b_2)_n \, (-1)^{(m k_1 + n k_2) } \, (-t_1)_{(m + 1)k_1} \, (-t_2)_{nk_2}}{ (c)_{m + n} \, m! \, n!} \ x^{m} \, y^n\nonumber\\
	& = k_1 \, (-1)^{k_1} \, (-t_1)_{k_1} \, x \, \rho_{t_1}^{k_1} \,  \left(\frac{1}{k_1} \Theta_{t_1} + a\right) \, \left(\frac{1}{k_1} \Theta_{t_1}  + b_1\right) \, \mathcal{F}^{(1)}_3.
\end{align}
Thus, we arrive at \eqref{1.15}. Similarly the equation \eqref{1.16} can be proved.  
\end{proof}

\subsection{Integral representations}
In this section, we find the integral representations for discrete Appell  function $\mathcal{F}^{(1)}_3$.
\begin{theorem}\label{t7}
	Let $a_1$, $a_2$, $b_1$, $b_2$, $c$, $t_1$ and $t_2$ be complex numbers. Then for $\vert x\vert < 1$, $\vert y\vert < 1$, the discrete Appell function $\mathcal{F}^{(1)}_3$ can be represented in the integral forms as
	\begin{align}
	& \mathcal{F}^{(1)}_3(a_1, a_2, b_1, b_2; c; t_1, t_2, k_1, k_2, x, y)\nonumber\\
	& =\Gamma \left(\begin{array}{c}
		c\\
		b_1, b_2, c - b_1 - b_2
	\end{array}\right) \iint u^{b_1 - 1} v^{b_2 - 1} (1-u-v)^{c -  b_1 - b_2 - 1} \nonumber\\
& \quad \times {}_1\mathcal{F}_0(a_1; - ; t_1,  k_1, u x) \, {}_1 \mathcal{F}_0(a_2; - ; t_2,  k_2, v y) \, du \, dv,	u\geq 0, \, v\geq 0, \, 1-u-v\geq 0. \label{1.18}
\end{align}
\begin{align}
	& \mathcal{F}^{(1)}_3(a_1, a_2, b_1, b_2; c; t_1, t_2, k_1, k_2, x, y)\nonumber\\
	& = \frac{1}{\Gamma (a_1)} \int_{0}^{\infty} e^{-u} \, u^{a_1 - 1} \nonumber\\
	& \quad \times F_{{1}:{0}; {0}} ^{{0}:{k_1 + 1}; {k_2 + 2}}\left(\begin{array}{ccc}
		- : & b_1, \frac{- t_1}{k_1}, \dots, \frac{- t_1 + k_1 - 1}{k_1} ; & a_2,  b_2, \frac{- t_2}{k_2}, \dots, \frac{- t_2 + k_2 - 1}{k_2}\\
		c: & - ; & - 
	\end{array}; (-k_1)^{k_1} \, u x,  (-k_2)^{k_2} \, y\right) du;\label{3.9}\\
	& = \frac{1}{\Gamma (a_2)} \int_{0}^{\infty} e^{-v} \, v^{a_2 - 1} \nonumber\\
	& \quad \times F_{{1}:{0}; {0}} ^{{0}:{k_1 + 2}; {k_2 + 1}}\left(\begin{array}{ccc}
		- : & a_1, b_1, \frac{- t_1}{k_1}, \dots, \frac{- t_1 + k_1 - 1}{k_1} ; &  b_2, \frac{- t_2}{k_2}, \dots, \frac{- t_2 + k_2 - 1}{k_2}\\
		c: & - ; & - 
	\end{array}; (-k_1)^{k_1} \,  x,  (-k_2)^{k_2} \, v y\right) dv;\label{e3.9}\\
	& = \frac{1}{\Gamma (b_1)} \int_{0}^{\infty} e^{-u} \, u^{b_1 - 1} \nonumber\\
	& \quad \times F_{{1}:{0}; {0}} ^{{0}:{k_1 +1}; {k_2 + 2}}\left(\begin{array}{ccc}
		- : & a_1, \frac{- t_1}{k_1}, \dots, \frac{- t_1 + k_1 - 1}{k_1} ; & a_2, b_2, \frac{- t_2}{k_2}, \dots, \frac{- t_2 + k_2 - 1}{k_2}\\
		c: & - ; & - 
	\end{array}; (-k_1)^{k_1} \, u x,  (-k_2)^{k_2} \, y\right) du;\\
	& = \frac{1}{\Gamma (b_2)} \int_{0}^{\infty} e^{-v} \, v^{b_2 - 1} \nonumber\\
	& \quad \times F_{{1}:{0}; {0}} ^{{1}:{k_1 + 2}; {k_2 + 1}}\left(\begin{array}{ccc}
		- : & a_1, b_1, \frac{- t_1}{k_1}, \dots, \frac{- t_1 + k_1 - 1}{k_1} ; & a_2, \frac{- t_2}{k_2}, \dots, \frac{- t_2 + k_2 - 1}{k_2}\\
		c: & - ; & - 
	\end{array}; (-k_1)^{k_1} \, x,  (-k_2)^{k_2} \, v y\right) dv;\\
& = \frac{1}{\Gamma (-t_1)} \int_{0}^{\infty} e^{-u} \, u^{-t_1 - 1} \nonumber\\
& \quad \times F_{{1}:{0}; {0}} ^{{0}:{2}; {k_2 + 2}}\left(\begin{array}{ccc}
- : & a_1, b_1 ; & a_2, b_2, \frac{- t_2}{k_2}, \dots, \frac{- t_2 + k_2 - 1}{k_2}\\
c: & - ; & - 
\end{array}; (- u)^{k_1} \,  x,  (-k_2)^{k_2} \, y\right) du;\\
& = \frac{1}{\Gamma (-t_2)} \int_{0}^{\infty} e^{-v} \, v^{-t_2 - 1} \nonumber\\
& \quad \times F_{{1}:{0}; {0}} ^{{1}:{k_1 + 2}; {2}}\left(\begin{array}{ccc}
- : & a_1, b_1, \frac{- t_1}{k_1}, \dots, \frac{- t_1 + k_1 - 1}{k_1} ; & a_2, b_2\\
c: & - ; & - 
\end{array}; (-k_1)^{k_1} \, x,  (- v)^{k_2} \, y\right) dv.
\end{align}
\end{theorem}	
\begin{proof}
To prove \eqref{1.18}, we use the integral formula
\begin{align}
	&\frac{(b_1)_m \, (b_2)_n}{(c)_{m + n}}\nonumber\\
	&  = \Gamma \left(\begin{array}{c}
		c\\
		b_1, b_2, c - b_1 - b_2
	\end{array}\right) \iint u^{b_1 + m - 1} v^{b_2 + n - 1} (1-u-v)^{c -  b_1 - b_2 - 1} \, du \, dv,
\end{align}
and as such get
\begin{align}
	& \mathcal{F}^{(1)}_3(a_1, a_2, b_1, b_2; c; t_1, t_2, k_1, k_2, x, y)\nonumber\\
	& =\Gamma \left(\begin{array}{c}
		c\\
		b_1, b_2, c - b_1 - b_2
	\end{array}\right) \sum_{m, n \geq 0} \iint u^{b_1 + m - 1} v^{b_2 + n - 1} (1-u-v)^{c -  b_1 - b_2 - 1} \nonumber\\
& \quad \times \frac{(a_1)_{m} \, (a_2)_n \, (-1)^{m k_1} \, (-t_1)_{mk_1} \, (-1)^{n k_2} \, (-t_2)_{nk_2} }{m! \, n!} \, x^m \, y^n \, du \, dv\nonumber\\
	& =\Gamma \left(\begin{array}{c}
	c\\
	b_1, b_2, c - b_1 - b_2
\end{array}\right) \sum_{m, n \geq 0} \iint u^{b_1 - 1} v^{b_2 - 1} (1-u-v)^{c -  b_1 - b_2 - 1} \nonumber\\
&\quad \times {}_1\mathcal{F}_0(a_1; - ; t_1,  k_1, u x) \, {}_1 \mathcal{F}_0(a_2; - ; t_2,  k_2, v y) du \, dv. 
\end{align}
It completes the proof. To prove \eqref{3.9}, we use the identity $(a)_{m} = \frac{\Gamma (a + m)}{\Gamma (a)}$ and the integral of $\Gamma (a + m)$ as
\begin{align}
	\Gamma (a + m) = \int_{0}^{\infty} e^{-u} \, u^{a + m - 1} \, du.
\end{align}
Similarly, the  remaining integrals can be determined.  
\end{proof}
Among the seven Humbert functions $\phi_1$, $\phi_2$, $\phi_3$, $\psi_1$, $\psi_2$, $\varXi_1$ and $\varXi_2$, the two degenerations $\varXi_1$ and $\varXi_2$ are obtained by taking limits on $F_3$. In parallel, when the limit is applied on $\mathcal{F}^{(1)}_3$, we get the first discrete form of $\varXi_1$ and $\varXi_2$ denoted by $\varXi_1^{(1)}$ and $\varXi_2^{(1)}$. We define the discrete Humbert functions $\varXi_1^{(1)}$ and $\varXi_2^{(1)}$ as:
\begin{align}
	&\varXi^{(1)}_1 \left(a_1, a_2, b_1; c; t_1, t_2, k_1, k_2, x,  y\right) \nonumber\\
	& =	\sum_{m,n \geq 0} \frac{(a_1)_{m} \, (a_2)_n \, (b_1)_m \, (-1)^{m k_1} \, (-t_1)_{mk_1} \, (-1)^{n k_2} \, (-t_2)_{nk_2}}{ (c)_{m+n} \, m! \, n!} \ x^m \, y^n;
\end{align}
\begin{align}
	&\varXi^{(1)}_2 \left(a_1, b_1; c; t_1, t_2, k_1, k_2, x,  y\right)\nonumber\\
	& = 	 \sum_{m,n \geq 0} \frac{(a_1)_{m} \, (b_1)_m \, (-1)^{m k_1} \, (-t_1)_{mk_1} \, (-1)^{n k_2} \, (-t_2)_{nk_2}}{(c)_{m+n} \, m! \, n!} \ x^m \, y^n.
\end{align}
Indeed, it can be verified that
\begin{align}
	& \lim_{\varepsilon \to 0}  \mathcal{F}^{(1)}_3 \left(a_1, a_2, b_1, \frac{1}{\varepsilon}; c; t_1, t_2, k_1, k_2, x, \varepsilon \, y\right) 
	 = \varXi^{(1)}_1 \left(a_1, a_2, b_1; c; t_1, t_2, k_1, k_2, x,  y\right),
\end{align}
\begin{align}
	&\lim_{\varepsilon \to 0}  \mathcal{F}^{(1)}_3 \left(a_1, \frac{1}{\varepsilon}, b_1,\frac{1}{\varepsilon}; c; t_1, t_2, k_1, k_2, x, \varepsilon^2 \, y\right)  = \varXi^{(1)}_2 \left(a_1, b_1; c; t_1, t_2, k_1, k_2, x,  y\right).
\end{align}
\section{Differential and difference formulae}
For difference operator $\Delta_{t}$ defined by $\Delta_{t} f(t) = f(t + 1) - f(t)$, we have $\Delta_{t} [(-1)^k \, (-t)_k] = k \, (-1)^{k - 1} \, (-t)_{k - 1}$. For differential operators $\theta = x \, \frac{\partial}{\partial x}, \phi = y \, \frac{\partial}{\partial y}$, we have the following theorem: 
\begin{theorem}
For $r \in \mathbb{N}$, following difference and differential formulae are satisfied by discrete Appell function $\mathcal{F}^{(1)}_3$:
\begin{align}
&	(\Delta_{t_1})^r \mathcal{F}^{(1)}_3(a_1, a_2, b_1, b_2; c; t_1, t_2, 1, k_2, x, y) \nonumber\\
& = \frac{(a_1)_r \, (b_1)_r \,  x^r }{(c)_r}  \mathcal{F}^{(1)}_3(a_1 + r, a_2, b_1 + r, b_2; c + r; t_1, t_2, 1, k_2, x, y);\label{4.1}\\
& (\Delta_{t_2})^r \mathcal{F}^{(1)}_3(a_1, a_2, b_1, b_2; c; t_1, t_2, k_1, 1, x, y) \nonumber\\
& = \frac{(a_2)_r \, (b_2)_r \,  y^r }{(c)_r}   \mathcal{F}^{(1)}_3(a_1, a_2 + r, b_1, b_2 + r; c + r; t_1, t_2 , k_1, 1, x, y);\label{4.2}\\
& (\theta)^r \mathcal{F}^{(1)}_3(a_1, a_2, b_1, b_2; c; t_1, t_2, k_1, k_2, x, y) \nonumber\\
& = \frac{(-1)^{rk_1} \, (a_1)_r \, (b_1)_r \, (-t_1)_{rk_1} \, x^r}{(c)_r} \nonumber\\
& \quad \times  \mathcal{F}^{(1)}_3(a_1 + r, a_2, b_1 + r, b_2; c + r; t_1 - rk_1, t_2, k_1, k_2, x, y);\\
& (\phi)^r \mathcal{F}^{(1)}_3(a_1, a_2, b_1, b_2; c; t_1, t_2, k_1, k_2, x, y) \nonumber\\
& = \frac{(-1)^{rk_2} \, (a_2)_r \, (b_2)_r \, (-t_2)_{rk_2} \, y^r}{(c)_r} \nonumber\\
& \quad \times  \mathcal{F}^{(1)}_3(a_1, a_2 + r, b_1, b_2 + r; c + r; t_1, t_2 - rk_2, k_1, k_2, x, y).
\end{align}
\end{theorem}
\begin{proof}
 The proofs of these differential formulae are very elementary and hence we give here only the proof of \eqref{4.1}. The rest can be verified in a similar way. The action of difference operator $\Delta_{t_1}$ on the discrete function $\mathcal{F}^{(1)}_3$ yields
 \begin{align}
 &(\Delta_{t_1}) \mathcal{F}^{(1)}_3(a, b_1, b_2; c; t_1, t_2, 1, k_2, x, y)\nonumber\\
& = \sum_{m,n \geq 0} \frac{(a_1)_{m} \, (a_2)_n \, (b_1)_m \, (b_2)_n \, (-1)^{m - 1} \, m \, (-t_1)_{m - 1} \, (-1)^{nk_2} \, (-t_2)_{nk_2}}{ (c)_{m+n} \, m! \, n!} \ x^m \, y^n\nonumber\\
& = \sum_{m, n \geq 0} \frac{(a_1)_{m + 1} \, (a_2)_n \, (b_1)_{m + 1} \, (b_2)_n \, (-1)^{m}  \, (-t_1)_{m} \, (-1)^{nk_2} \, (-t_2)_{nk_2}}{ (c)_{m+n + 1} \, m! \, n!} \ x^{m + 1} \, y^n\nonumber\\
& =   x \, \frac{a_1 \, b_1}{c}\nonumber\\
& \quad \times \sum_{m,n \geq 0} \frac{(a_1 + 1)_{m} \, (a_2)_n \, (b_1 + 1)_m \, (b_2)_n \, (-1)^{m} \, (-t_1)_{m} \, (-1)^{nk_2} \, (-t_2)_{nk_2}}{ (c)_{m+n} \, m! \, n!} \ x^m \, y^n\nonumber\\
& =   x \, \frac{a_1 \, b_1}{c} \mathcal{F}^{(1)}_3(a_1 + 1, a_2, b_1 + 1, b_2; c + 1; t_1, t_2, 1, k_2, x, y).
 \end{align}
Now, acting the difference operator twice leads to
  \begin{align}
  	&(\Delta_{t_1})^2 \mathcal{F}^{(1)}_3(a_1, a_2, b_1, b_2; c; t_1, t_2, 1, k_2, x, y)\nonumber\\
  	& = x^2 \, \frac{(a_1)_2 \, (b_1)_2}{(c)_2} \, \mathcal{F}^{(1)}_3(a_1 + 2, a_2, b_1 + 2, b_2; c + 2; t_1, t_2, 1, k_2, x, y).
  \end{align}
Inductively, we get
\begin{align}
	&	(\Delta_{t_1})^r \mathcal{F}^{(1)}_3(a_1, a_2, b_1, b_2; c; t_1, t_2, 1, k_2, x, y) \nonumber\\
	& = \frac{(a_1)_r \, (b_1)_r \,  x^r}{(c)_r}   \mathcal{F}^{(1)}_3(a_1 + r, a_2, b_1 + r, b_2; c + r; t_1, t_2, 1, k_2, x, y).
\end{align}
It completes the proof. 
\end{proof}
Some other differential formulas obeyed by $\mathcal{F}^{(1)}_3$ are as follows
\begin{align}
	& \left(\frac{\partial}{\partial x}\right)^r \left[x^{b_1 + r - 1} \mathcal{F}^{(1)}_3(a_1, a_2, b_1, b_2; c; t_1, t_2, k_1, k_2, x, y)\right]\nonumber\\
	& = x^{b_1 - 1} \, (b_1)_r \mathcal{F}^{(1)}_3(a_1, a_2, b_1 + r, b_2; c; t_1, t_2, k_1, k_2, x, y);\label{4.14}\\
		& \left(\frac{\partial}{\partial y}\right)^r [y^{b_2 + r - 1} \mathcal{F}^{(1)}_3(a_1, a_2, b_1, b_2; c; t_1, t_2, k_1, k_2, x, y)]\nonumber\\
	& = y^{b_2 - 1} \, (b_2)_r \mathcal{F}^{(1)}_3(a_1, a_2, b_1, b_2 + r; c; t_1, t_2, k_1, k_2, x, y);\\
		& \left(\frac{\partial}{\partial x}\right)^r [x^{a_1 + r - 1} \mathcal{F}^{(1)}_3(a_1, a_2, b_1, b_2; c; t_1, t_2, k_1, k_2, x, y)]\nonumber\\
	& = x^{a_1 - 1} \, (a_1)_r \mathcal{F}^{(1)}_3(a_1 + r, a_2, b_1, b_2; c; t_1, t_2, k_1, k_2, x, y);\\
		& \left(\frac{\partial}{\partial y}\right)^r [y^{a_2 + r - 1} \mathcal{F}^{(1)}_3(a_1, a_2, b_1, b_2; c; t_1, t_2, k_1, k_2, x, y)]\nonumber\\
	& = y^{a_2 - 1} \, (a_2)_r \mathcal{F}^{(1)}_3(a_1, a_2 + r, b_1, b_2; c; t_1, t_2, k_1, k_2, x, y);\\
		& \left(\frac{\partial}{\partial x}\right)^r [x^{c - 1} \mathcal{F}^{(1)}_3(a, b_1, b_2; c; t_1, t_2, k_1, k_2, x, xy)]\nonumber\\
	& = (-1)^r \, (1 - c)_r \, x^{c - r - 1} \mathcal{F}^{(1)}_3(a_1, a_2, b_1, b_2; c - r; t_1, t_2, k_1, k_2, x, xy);\\
		& \left(\frac{\partial}{\partial y}\right)^r [y^{c - 1} \mathcal{F}^{(1)}_3(a_1, a_2, b_1, b_2; c; t_1, t_2, k_1, k_2, xy, y)]\nonumber\\
	& = (-1)^r \, y^{c - r - 1} \, (1 - c)_r \mathcal{F}^{(1)}_3(a_1, a_2, b_1, b_2; c - r; t_1, t_2, k_1, k_2, xy, y).
\end{align}
Proofs of these formulae can be done using induction.
\section{Finite and infinite sum}
Now, we establish some finite and infinite summation formulae in terms of discrete Appell function $\mathcal{F}^{(1)}_3$ in the following theorem.
\begin{theorem} 
	Following summation formulas hold true
	\begin{align}
& \mathcal{F}^{(1)}_3(a_1, a_2, b_1 + r, b_2; c; t_1, t_2, k_1, k_2, x, y)\nonumber\\
& = \sum_{s = 0}^{r} {r \choose s} \frac{(a_1)_s \, (-1)^{sk_1} \, (-t_1)_{sk_1}}{(c)_s} \, x^s \, \mathcal{F}^{(1)}_3(a_1 + s, a_2, b_1 + s, b_2; c + s; t_1 - sk_1, t_2, k_1, k_2, x, y);\label{5.1}\\
& \mathcal{F}^{(1)}_3(a_1, a_2, b_1, b_2 + r; c; t_1, t_2, k_1, k_2, x, y)\nonumber\\
& = \sum_{s = 0}^{r} {r \choose s} \frac{(a_2)_s \, (-1)^{sk_2} \, (-t_2)_{sk_2}}{(c)_s} \, y^s \, \mathcal{F}^{(1)}_3(a_1, a_2 + s, b_1, b_2 + s; c + s; t_1, t_2 - sk_2, k_1, k_2, x, y);\label{5.2}\\
	&\sum_{r = 0}^{\infty} \frac{(a_1)_r}{r !} \, z^r \, \mathcal{F}^{(1)}_3 (a_1 + r, a_2, b_1, b_2; c ; t_1, t_2, k, x, y)\nonumber\\
	& = (1 - z)^{-a_1} \, \mathcal{F}^{(1)}_3 \left(a_1, a_2, b_1, b_2; c; t_1, t_2, k, \frac{x}{1 - z}, y\right). \label{5.3}
	\end{align}
\end{theorem}  
\begin{proof}
To prove \eqref{5.1}, we start by applying the Leibnitz rule and get
\begin{align}
	& \left(\frac{\partial}{\partial x}\right)^r \left[x^{b_1 + r - 1} \, \mathcal{F}^{(1)}_3(a_1, a_2, b_1, b_2; c; t_1, t_2, k_1, k_2, x, y)\right]\nonumber\\
	& = \sum_{s = 0}^{r} {r \choose s} \left[\left(\frac{\partial}{\partial x}\right)^{r - s} \, x^{b_1 + r - 1}\right] \, \left[\left(\frac{\partial}{\partial x}\right)^s  \, \mathcal{F}^{(1)}_3(a_1, a_2, b_1, b_2; c; t_1, t_2, k_1, k_2, x, y)\right]\nonumber\\
	& = \sum_{s = 0}^{r} {r \choose s} (-1)^{r - s} \, (1 - b_1 + s)_{r - s} \, x^{b_1 + s - 1} \, \frac{(a_1)_s \, (b_1)_s \, (-1)^{sk_1} \, (-t_1)_{sk_1}}{(c)_s}\nonumber\\
	& \quad \times \mathcal{F}^{(1)}_3(a_1 + s, a_2, b_1 + s, b_2; c + s; t_1 - sk_1, t_2, k_1, k_2, x, y) \nonumber\\
	& = \sum_{s = 0}^{r} {r \choose s} \, (-1)^{r - s} \frac{(-1)^s \, (1 - b_1 - r)_r}{(b_1)_s} \, x^{b_1 + s - 1} \, \frac{(a_1)_s \, (b_1)_s \, (-1)^{sk_1} \, (-t_1)_{sk_1}}{(c)_s}\nonumber\\
	& \quad \times \mathcal{F}^{(1)}_3(a_1 + s, a_2, b_1 + s, b_2; c + s; t_1 - sk_1, t_2, k_1, k_2, x, y) \nonumber\\
	& = (b_1)_r \, \sum_{s = 0}^{r} {r \choose s} \,  x^{b_1 + s - 1} \, \frac{(a_1)_s  \, (-1)^{sk_1} \, (-t_1)_{sk_1}}{(c)_s}\nonumber\\
	& \quad \times \mathcal{F}^{(1)}_1(a_1 + s, a_2, b_1 + s, b_2; c + s; t_1 - sk_1, t_2, k_1, k_2, x, y).
\end{align}
Taking into account the equation \eqref{4.14}, we have
\begin{align}
&	x^{b_1 - 1} \, (b_1)_r \, \mathcal{F}^{(1)}_3(a_1, a_2, b_1 + r, b_2; c; t_1, t_2, k_1, k_2, x, y)\nonumber\\
	& = (b_1)_r \, \sum_{s = 0}^{r} {r \choose s} \,  x^{b_1 + s - 1} \, \frac{(a_1)_s  \, (-1)^{sk_1} \, (-t_1)_{sk_1}}{(c)_s}\nonumber\\
	& \quad \times \mathcal{F}^{(1)}_3(a_1, a_2 + s, b_1 + s, b_2; c + s; t_1 - sk_1, t_2, k_1, k_2, x, y).
\end{align}
This implies
	\begin{align}
	& \mathcal{F}^{(1)}_3 (a_1, a_2, b_1 + r, b_2; c; t_1, t_2, k_1, k_2, x, y)\nonumber\\
	& = \sum_{s = 0}^{r} {r \choose s} \frac{(a_1)_s \, (-1)^{sk_1} \, (-t_1)_{sk_1}}{(c)_s} \, x^s \, \mathcal{F}^{(1)}_3(a_1 + s, a_2, b_1 + s, b_2; c + s; t_1 - sk_1, t_2, k_1, k_2, x, y).
\end{align}
This completes the proof. A similar procedure can be used to get equation \eqref{5.2}. To prove \eqref{5.3}, we start with the right hand side of the equation and thus get
\begin{align}
	&(1 - z)^{-a_1} \, \mathcal{F}^{(1)}_3 \left(a_1, a_2, b_1, b_2; c; t_1, t_2, k_1, k_2, \frac{x}{1 - z}, y\right) \nonumber\\
	& = \sum_{m,n \geq 0} (1 - z)^{- (a_1 + m)} \, \frac{(a_1)_{m} \, (a_2)_n \, (b_1)_m \, (b_2)_n \, (-1)^{m k_1} \, (-t_1)_{mk_1} \, (-1)^{nk_2} (-t_2)_{nk_2}}{ (c)_{m+n} \, m! \, n!} \ x^m \, y^n.
\end{align}
Using the binomial theorem $(1 - z)^{- (a_1 + m)} = \sum_{r = 0}^{\infty} \frac{(a_1 + m)_r}{r !} \, z^r$ and identity $(a_1)_{m + r} = (a_1)_{m} \, (a_1 + m)_r = (a_1)_r \, (a_1 + r)_{m}$, we have
\begin{align}
	&(1 - z)^{-a_1} \, \mathcal{F}^{(1)}_3 \left(a_1, a_2, b_1, b_2; c; t_1, t_2, k_1, k_2, \frac{x}{1 - z}, y\right) \nonumber\\
	& = \sum_{m,n \geq 0}  \, \sum_{r = 0}^{\infty} \frac{(a_1)_r}{r !} \, z^r  \frac{(a_1 + r)_{m} \, (a_2)_n (b_1)_m \, (b_2)_n \, (-1)^{mk_1} \, (-t_1)_{mk_1} \, (-1)^{nk_2} \, (-t_2)_{nk_2}}{ (c)_{m+n} \, m! \, n!} \ x^m \, y^n\nonumber\\
	& = \sum_{r = 0}^{\infty} \frac{(a_1)_r}{r !} \, z^r \, \mathcal{F}^{(1)}_3 (a_1 + r, a_2, b_1, b_2; c ; t_1, t_2, k_1, k_2, x, y).
\end{align}
It completes the proof.
\end{proof}

\section{Recursion Formulae}
We have the following theorem:
\begin{theorem} 
Following recursion formulas hold true for $\mathcal{F}^{(1)}_3$:
\begin{align}
& \mathcal{F}^{(1)}_3 (a_1 + s, a_2, b_1, b_2; c; t_1, t_2, k_1, k_2, x, y) \nonumber\\
& = \mathcal{F}^{(1)}_3 (a_1, a_2, b_1, b_2; c ; t_1, t_2, k_1, k_2, x, y)\nonumber\\
& \quad  + \frac{(-1)^{k_1} \, (-t_1)_{k_1} \, b_1 \, x}{c}  \sum_{r = 1}^{s} \mathcal{F}^{(1)}_3 (a_1 + r, a_2, b_1 + 1, b_2; c + 1; t_1 - k_1, t_2, k_1, k_2, x, y);\label{e6.1}\\
 & \mathcal{F}^{(1)}_3 (a_1 - s, a_2, b_1, b_2; c ; t_1, t_2, k_1, k_2, x, y) \nonumber\\
& = \mathcal{F}^{(1)}_3 (a_1, a_2, b_1, b_2; c ; t_1, t_2, k_1, k_2, x, y)\nonumber\\
& \quad  - \frac{(-1)^{k_1} \, (-t_1)_{k_1} \, b_1 \, x}{c} \sum_{r = 0}^{s - 1} \mathcal{F}^{(1)}_3 (a_1 - r, a_2, b_1 + 1, b_2; c + 1; t_1 - k_1, t_2, k_1, k_2, x, y);\\
& \mathcal{F}^{(1)}_3 (a_1, a_2 + s, b_1, b_2; c; t_1, t_2, k_1, k_2, x, y) \nonumber\\
& = \mathcal{F}^{(1)}_3 (a_1, a_2, b_1, b_2; c ; t_1, t_2, k_1, k_2, x, y)\nonumber\\
&\quad  + \frac{(-1)^{k_2} \, (-t_2)_{k_2} \, b_2 \, y}{c}  \sum_{r = 1}^{s} \mathcal{F}^{(1)}_3 (a_1, a_2 + r, b_1, b_2 + 1; c + 1; t_1, t_2 - k_2, k_1, k_2, x, y);\\
& \mathcal{F}^{(1)}_3 (a_1, a_2 - s, b_1, b_2; c ; t_1, t_2, k_1, k_2, x, y) \nonumber\\
& = \mathcal{F}^{(1)}_3 (a_1, a_2, b_1, b_2; c ; t_1, t_2, k_1, k_2, x, y)\nonumber\\
& \quad  - \frac{(-1)^{k_2} \, (-t_2)_{k_2} \, b_2 \, y}{c}  \sum_{r = 0}^{s - 1} \mathcal{F}^{(1)}_3 (a_1, a_2 - r, b_1, b_2 + 1; c + 1; t_1, t_2 - k_2, k_1, k_2, x, y);\\
& \mathcal{F}^{(1)}_3 (a_1, a_2, b_1 + s, b_2; c ; t_1, t_2, k_1, k_2, x, y) \nonumber\\
& = \mathcal{F}^{(1)}_3 (a_1, a_2, b_1, b_2; c ; t_1, t_2, k_1, k_2, x, y)\nonumber\\
& \quad  + \frac{(-1)^{k_1} \, (-t_1)_{k_1} \, a \, x}{c}  \sum_{r = 1}^{s} \mathcal{F}^{(1)}_3 (a_1 + 1, a_2, b_1 + r, b_2; c + 1; t_1 - k_1, t_2, k_1, k_2, x, y);\\
& \mathcal{F}^{(1)}_3 (a_1, a_2, b_1 - s, b_2; c ; t_1, t_2, k_1, k_2, x, y) \nonumber\\
& = \mathcal{F}^{(1)}_3 (a_1, a_2, b_1, b_2; c ; t_1, t_2, k_1, k_2, x, y)\nonumber\\
& \quad  - \frac{(-1)^{k_1} \, (-t_1)_{k_1} \, a_1 \, x}{c}  \sum_{r = 0}^{s - 1} \mathcal{F}^{(1)}_3 (a_1 + 1, b_1 - r, b_2; c + 1; t_1 - k_1, t_2, k_1, k_2, x, y);\\
& \mathcal{F}^{(1)}_3 (a_1, a_2, b_1, b_2; c - s; t_1, t_2, k_1, k_2, x, y) \nonumber\\
& = \mathcal{F}^{(1)}_3 (a_1, a_2, b_1, b_2; c ; t_1, t_2, k_1, k_2, x, y) \nonumber\\
& \quad  + (-1)^{k_1} \, (-t_1)_{k_1} \, a_1 \, b_1 \, x  \sum_{r = 1}^{s} \frac{\mathcal{F}^{(1)}_3 (a_1 + 1, a_2, b_1 + 1, b_2; c + 2 - r; t_1 - k_1, t_2, k_1, k_2, x, y)}{(c - r) \, (c - r + 1)}\nonumber\\
& \quad + (-1)^{k_2} \, (-t_2)_{k_2} \, a_2 \, b_2 \, y  \sum_{r = 1}^{s} \frac{\mathcal{F}^{(1)}_3 (a_1, a_2 + 1, b_1, b_2 + 1; c + 2 - r; t_1, t_2 - k_2, k_1, k_2, x, y)}{(c - r) \, (c - r + 1)}.
\end{align}
\end{theorem}
\begin{proof}
To prove the formula \eqref{e6.1}, we begin with
\begin{align}
 &\mathcal{F}^{(1)}_3 (a_1, a_2, b_1, b_2; c ; t_1, t_2, k_1, k_2, x, y) + \frac{(-1)^{k_1} \, (-t_1)_{k_1} \, b_1 \, x}{c} \nonumber\\
& \quad \times  \mathcal{F}^{(1)}_3 (a_1 + 1, a_2, b_1 + 1, b_2; c + 1; t_1 - k_1, t_2, k_1, k_2, x, y) \nonumber\\
& = \sum_{m,n \geq 0}  \, \frac{(a_1)_{m} \, (a_2)_n \, (b_1)_m \, (b_2)_n \, (-1)^{m k_1} \, (-t_1)_{mk_1} \, (-1)^{nk_2} \, (-t_2)_{nk_2}}{ (c)_{m+n} \, m! \, n!} \ x^m \, y^n + \frac{(-1)^{k_1} \, (-t_1)_{k_1} \, b_1 \, x}{c} \nonumber\\
& \quad \times \sum_{m,n \geq 0}  \, \frac{(a_1 + 1)_{m} \, (a_2)_n \, (b_1 + 1)_m \, (b_2)_n \, (-1)^{mk_1} \, (-t_1 + k_1)_{mk_1} \, (-1)^{nk_2} \, (-t_2)_{nk_2}}{ (c + 1)_{m+n} \, m! \, n!} \ x^m \, y^n\nonumber\\
& = \sum_{m,n \geq 0}  \, \frac{(a_1)_{m} \, (a_2)_n \, (b_1)_m \, (b_2)_n \, (-1)^{mk_1} \, (-t_1)_{mk_1} \, (-1)^{nk_2} \, (-t_2)_{nk_2}}{ (c)_{m+n} \, m! \, n!} \ x^m \, y^n \nonumber\\
& \quad + \sum_{m,n \geq 0} \frac{m}{a_1} \, \frac{(a_1)_{m} \, (a_2)_n \, (b_1)_m \, (b_2)_n \, (-1)^{mk_1} \, (-t_1)_{mk_1} \, (-1)^{nk_2} \, (-t_2)_{nk_2}}{ (c)_{m+n} \, m! \, n!} \ x^m \, y^n \nonumber\\
& =  \sum_{m,n \geq 0}  \frac{a_1 + m }{a_1} \, \frac{(a_1)_{m} \, (a_2)_n \, (b_1)_m \, (b_2)_n \, (-1)^{mk_1} \, (-t_1)_{mk_1} \, (-1)^{nk_2} \, (-t_2)_{nk_2}}{ (c)_{m+n} \, m! \, n!} \ x^m \, y^n\nonumber\\
& = \mathcal{F}^{(1)}_3 (a_1 + 1, a_2, b_1, b_2; c; t_1, t_2, k_1, k_2, x, y). 
\end{align}
Again
\begin{align}
	&\mathcal{F}^{(1)}_3 (a_1, a_2, b_1, b_2; c ; t_1, t_2, k_1, k_2, x, y) + \frac{(-1)^{k_1} \, (-t_1)_{k_1} \, b_1 \, x}{c} \nonumber\\
	& \quad \times  [\mathcal{F}^{(1)}_3 (a_1 + 1, a_2, b_1 + 1, b_2; c + 1; t_1 - k_1, t_2, k_1, k_2, x, y) \nonumber\\
	& \quad  + \mathcal{F}^{(1)}_3 (a + 2, b_1 + 1, b_2; c + 1; t_1 - k_1, t_2, k_1, k_2, x, y)] \nonumber\\
	& =  \sum_{m,n \geq 0} \frac{(a_1)_{m} \, (a_2)_n \, (b_1)_m \, (b_2)_n \, (-1)^{mk_1} \, (-t_1)_{mk_1} \, (-1)^{nk_2} \, (-t_2)_{nk_2}}{ (c)_{m+n} \, m! \, n!} \ x^m \, y^n \nonumber\\
	& \quad + \frac{(-1)^{k_1} \, (-t_1)_{k_1} \, b_1 \, x}{c}  \sum_{m,n \geq 0} (a_1 + 2)_{m - 1} (2a_1 + m + 2) \nonumber\\
	& \quad \times \frac{(a_2)_n \, (b_1 + 1)_m \, (b_2)_n \, (-1)^{mk_1} \, (-t_1 + k_1)_{mk_1} \, (-1)^{nk_2} \, (-t_2)_{nk_2}}{ (c + 1)_{m+n} \, m! \, n!} \ x^m \, y^n\nonumber\\
	& = \sum_{m, n \ge 0} \frac{a_1 \, (a_1 + 1) + (2a_1 + m + 1) (m)}{a \, (a + 1)}\nonumber\\
	& \quad \times \frac{(a_1)_{m} \, (a_2)_n \, (b_1)_m \, (b_2)_n \, (-1)^{mk_1} \, (-t_1)_{mk_1} \, (-1)^{nk_2} \, (-t_2)_{nk_2}}{ (c)_{m+n} \, m! \, n!} \ x^m \, y^n\nonumber\\
	& = \frac{(a_1 + 2)_{m} \, (a_2)_n \, (b_1)_m \, (b_2)_n \, (-1)^{mk_1} \, (-t_1)_{mk_1} \, (-1)^{nk_2} \, (-t_2)_{nk_2}}{ (c)_{m+n} \, m! \, n!} \ x^m \, y^n\nonumber\\
	& = \mathcal{F}^{(1)}_3 (a_1 + 2, a_2, b_1, b_2; c; t_1, t_2, k_1, k_2, x, y).
\end{align}
We can generalize the result and get the formula \eqref{e6.1}. Hence the proof is completed. The other formulae can be verified in the same manner. 
\end{proof}
Now, we produce a list of difference and differential recursion formulas obeyed by discrete Appell function $\mathcal{F}^{(1)}_3$. Let $\theta = x \, \frac{\partial}{\partial x}$ and $\phi = y \, \frac{\partial}{\partial y}$ be the differential operators. Then, the following simple differential formulae can be easily verified: 
\begin{align}
&	a_1 \, \mathcal{F}^{(1)}_3 (a_1 + 1)  = (a_1 + \theta) \, \mathcal{F}^{(1)}_3;\\
& (a_1 + \theta - 1) \, \mathcal{F}^{(1)}_3 (a_1 - 1)  =	(a_1 - 1) \, \mathcal{F}^{(1)}_3;\\
&	a_2 \, \mathcal{F}^{(1)}_3 (a_2 + 1)  = (a_2 + \phi) \, \mathcal{F}^{(1)}_3;\\
& (a_2 + \phi - 1) \, \mathcal{F}^{(1)}_3 (a_2 - 1)  =	(a_2 - 1) \, \mathcal{F}^{(1)}_3;\\
&	b_1 \, \mathcal{F}^{(1)}_3 (b_1 + 1)  = (b_1 + \theta) \, \mathcal{F}^{(1)}_3;\\
&(b_1 + \theta - 1) \, \mathcal{F}^{(1)}_3 (b_1 - 1)  =	(b_1 - 1) \, \mathcal{F}^{(1)}_3;\\
&	b_2 \, \mathcal{F}^{(1)}_3 (b_2 + 1)  = (b_2 + \phi) \, \mathcal{F}^{(1)}_3;\\
& (b_2  + \phi - 1) \, \mathcal{F}^{(1)}_3 (b_2 - 1)  =	(b_2 - 1) \, \mathcal{F}^{(1)}_3;\\
&	(c - 1) \, \mathcal{F}^{(1)}_3 (c - 1)  = (c + \theta + \phi - 1) \, \mathcal{F}^{(1)}_3;\\
& (c + \theta + \phi) \, \mathcal{F}^{(1)}_3 (c + 1)  =	c \, \mathcal{F}^{(1)}_3;
\end{align}
where
\begin{align}
&\mathcal{F}^{(1)}_3 (a_1 \pm 1)  = \mathcal{F}^{(1)}_3 (a_1 \pm 1, a_2, b_1, b_2; c; t_1, t_2, k_1, k_2, x, y),
\end{align}
etc. Any two of the above differential recursion relations can be combined and hence produced a first or second order differential recursion relation as follows:
\begin{align}
& a_1 \, (a_1 - 1) \, \mathcal{F}^{(1)}_3 (a_1 + 1) - (a_1 + \theta) \, (a_1 + \theta - 1) \mathcal{F}^{(1)}_3 (a_1 - 1) = 0;\\
& a_1 \, (b_1 - 1) \, \mathcal{F}^{(1)}_3 (a_1 + 1) - (a_1 + \theta) \, (b_1 + \theta  - 1) \mathcal{F}^{(1)}_3 (b_1 - 1) = 0;\\
& a_1 \, (b_2 - 1) \, \mathcal{F}^{(1)}_3 (a_1 + 1) - (a_1 + \theta) \, (b_2 + \phi - 1) \mathcal{F}^{(1)}_3 (b_2 - 1) = 0;\\
& a_1 \, c \, \mathcal{F}^{(1)}_3 (a_1 + 1) - (a_1 + \theta) \, (c + \theta + \phi) \mathcal{F}^{(1)}_3 (c + 1) = 0;\\
	& a_1 \, (b_1 + \theta) \, \mathcal{F}^{(1)}_3 (a_1 + 1) - b_1 \, (a_1 + \theta) \, \mathcal{F}^{(1)}_3 (b_1 + 1) = 0;\\
		& a_1 \, (b_2 + \phi) \, \mathcal{F}^{(1)}_3 (a_1 + 1) - b_2 \, (a_1 + \theta) \, \mathcal{F}^{(1)}_3 (b_2 + 1) = 0;\\
	& a_1 \, (c + \theta + \phi - 1) \, \mathcal{F}^{(1)}_3 (a_1 + 1) - (c - 1) \, (a_1 + \theta) \, \mathcal{F}^{(1)}_3 (c - 1) = 0;\\	
	& a_1 \, (a_2 + \phi) \, \mathcal{F}^{(1)}_3 (a_1 + 1) - a_2 \, (a_1 + \theta) \, \mathcal{F}^{(1)}_3 (a_2 + 1) = 0;\\
	& a_1 \, (a_2 - 1) \, \mathcal{F}^{(1)}_3 (a_1 + 1) - (a_1 + \theta) \, (a_2 + \phi - 1) \, \mathcal{F}^{(1)}_3 (a_2 - 1) = 0;\\
	& (a_1 + \theta - 1) \, (a_2 + \phi) \, \mathcal{F}^{(1)}_3 (a_1 - 1) - a_2 \, (a_1 - 1) \mathcal{F}^{(1)}_3 (a_2 + 1) = 0;\\
	& (a_2 - 1) \, (a_1 + \theta - 1) \, \mathcal{F}^{(1)}_3 (a_1 - 1) - (a_1 - 1) \, (a_2 + \phi  - 1) \mathcal{F}^{(1)}_3 (a_2 - 1) = 0;\\
	& a_2 \, (a_2 - 1) \, \mathcal{F}^{(1)}_3 (a_2 + 1) - (a_2 + \phi) \, (a_2 + \phi - 1) \mathcal{F}^{(1)}_3 (a_2 - 1) = 0;\\
	& a_2 \, (b_1 - 1) \, \mathcal{F}^{(1)}_3 (a_2 + 1) - (a_2 + \phi) \, (b_1 + \theta  - 1) \mathcal{F}^{(1)}_3 (b_1 - 1) = 0;\\
	& a_2 \, (b_2 - 1) \, \mathcal{F}^{(1)}_3 (a_2 + 1) - (a_2 + \phi) \, (b_2 + \phi - 1) \mathcal{F}^{(1)}_3 (b_2 - 1) = 0;\\
	& a_2 \, c \, \mathcal{F}^{(1)}_3 (a_2 + 1) - (a_2 + \phi) \, (c + \theta + \phi) \mathcal{F}^{(1)}_3 (c + 1) = 0;\\
	& a_2 \, (b_1 + \theta) \, \mathcal{F}^{(1)}_3 (a_2 + 1) - b_1 \, (a_2 + \phi) \, \mathcal{F}^{(1)}_3 (b_1 + 1) = 0;\\
	& a_2 \, (b_2 + \phi) \, \mathcal{F}^{(1)}_3 (a_2 + 1) - b_2 \, (a_2 + \phi) \, \mathcal{F}^{(1)}_3 (b_2 + 1) = 0;\\
	& a_2 \, (c + \theta + \phi - 1) \, \mathcal{F}^{(1)}_3 (a_2 + 1) - (c - 1) \, (a_2 + \phi) \, \mathcal{F}^{(1)}_3 (c - 1) = 0;\\
 	&(a_1 + \theta - 1) \, (b_1 + \theta) \, \mathcal{F}^{(1)}_3 (a_1 - 1) - b_1 \, (a_1 - 1)  \, \mathcal{F}^{(1)}_3 (b_1 + 1) = 0;\\
 &	(a_1 + \theta - 1) \, (b_2 + \phi) \, \mathcal{F}^{(1)}_3 (a_1 - 1) - b_2 \, (a_1 - 1)  \, \mathcal{F}^{(1)}_3 (b_2 + 1) = 0;\\
 &	(a_1 + \theta - 1) \, (c + \theta + \phi - 1) \, \mathcal{F}^{(1)}_3 (a_1 - 1) - (c - 1) \, (a_1 - 1)  \, \mathcal{F}^{(1)}_3 (c - 1) = 0;\\
 & (b_1 - 1) \,	(a_1 + \theta - 1) \,  \mathcal{F}^{(1)}_3 (a_1 - 1) -  (a_1 - 1)  \, (b_1 + \theta - 1) \mathcal{F}^{(1)}_3 (b_1 - 1) = 0;\\
 & (b_2 - 1) \,	(a_1 + \theta - 1) \,  \mathcal{F}^{(1)}_3 (a_1 - 1) -  (a_1 - 1)  \, (b_2 + \phi - 1) \mathcal{F}^{(1)}_3 (b_2 - 1) = 0;\\
 & c \,	(a_1 + \theta - 1) \,  \mathcal{F}^{(1)}_3 (a_1 - 1) -  (a_1 - 1)  \, (c + \theta + \phi) \mathcal{F}^{(1)}_3 (c + 1) = 0;\\
  	&(a_2 + \phi - 1) \, (b_1 + \theta) \, \mathcal{F}^{(1)}_3 (a_2 - 1) - b_1 \, (a_2 - 1)  \, \mathcal{F}^{(1)}_3 (b_1 + 1) = 0;\\
 &	(a_2 + \phi - 1) \, (b_2 + \phi) \, \mathcal{F}^{(1)}_3 (a_2 - 1) - b_2 \, (a_2 - 1)  \, \mathcal{F}^{(1)}_3 (b_2 + 1) = 0;\\
 &	(a_2 + \phi - 1) \, (c + \theta + \phi - 1) \, \mathcal{F}^{(1)}_3 (a_2 - 1) - (c - 1) \, (a_2 - 1)  \, \mathcal{F}^{(1)}_3 (c - 1) = 0;\\
 & (b_1 - 1) \,	(a_2 + \phi - 1) \,  \mathcal{F}^{(1)}_3 (a_2 - 1) -  (a_2 - 1)  \, (b_1 + \theta - 1) \mathcal{F}^{(1)}_3 (b_1 - 1) = 0;\\
 & (b_2 - 1) \,	(a_2 + \phi - 1) \,  \mathcal{F}^{(1)}_3 (a_2 - 1) -  (a_2 - 1)  \, (b_2 + \phi - 1) \mathcal{F}^{(1)}_3 (b_2 - 1) = 0;\\
 & c \,	(a_2 + \phi - 1) \,  \mathcal{F}^{(1)}_3 (a_2 - 1) -  (a_2 - 1)  \, (c + \theta + \phi) \mathcal{F}^{(1)}_3 (c + 1) = 0;\\
 & b_1 \,	(b_1 - 1) \,  \mathcal{F}^{(1)}_3 (b_1 + 1) -  (b_1 + \theta)  \, (b_1 + \theta - 1) \mathcal{F}^{(1)}_3 (b_1 - 1) = 0;\\
 & b_1 \,	(b_2 + \phi) \,  \mathcal{F}^{(1)}_3 (b_1 + 1) -  b_2  \, (b_1 + \theta ) \mathcal{F}^{(1)}_3 (b_2 + 1) = 0;\\
  & b_1 \,	(b_2 - 1) \,  \mathcal{F}^{(1)}_3 (b_1 + 1) -  (b_1 + \theta)  \, (b_2 + \phi - 1) \mathcal{F}^{(1)}_3 (b_2 - 1) = 0;\\
  & b_1 \,	(c + \theta + \phi - 1) \,  \mathcal{F}^{(1)}_3 (b_1 + 1) -  (c - 1)  \, (b_1 + \theta) \mathcal{F}^{(1)}_3 (c - 1) = 0;\\
  & b_1 \,	c \,  \mathcal{F}^{(1)}_3 (b_1 + 1) -  (c + \theta +\phi)  \, (b_1 + \theta) \mathcal{F}^{(1)}_3 (c + 1) = 0;\\
   & b_2 \,	(b_1 - 1) \,  \mathcal{F}^{(1)}_3 (b_2 + 1) -  (b_2 + \phi)  \, (b_1 + \theta - 1) \mathcal{F}^{(1)}_3 (b_1 - 1) = 0;\\
  & b_2 \,	(b_2 - 1) \,  \mathcal{F}^{(1)}_3 (b_2 + 1) -  (b_2 + \phi)  \, (b_2 + \phi - 1) \mathcal{F}^{(1)}_3 (b_2 - 1) = 0;\\
  & b_2 \,	(c + \theta + \phi - 1) \,  \mathcal{F}^{(1)}_3 (b_2 + 1) -  (c - 1)  \, (b_2 + \phi) \mathcal{F}^{(1)}_3 (c - 1) = 0;\\
  & b_2 \,	c \,  \mathcal{F}^{(1)}_3 (b_2 + 1) -  (c + \theta +\phi)  \, (b_2 + \phi) \mathcal{F}^{(1)}_3 (c + 1) = 0;\\
  & (b_2 - 1) \,	(b_1 + \theta - 1) \,  \mathcal{F}^{(1)}_3 (b_1 - 1) -  (b_1 - 1)  \, (b_2 + \phi - 1) \mathcal{F}^{(1)}_3 (b_2 - 1) = 0;\\
  & (b_1 + \theta - 1) \,	(c + \theta + \phi - 1) \,  \mathcal{F}^{(1)}_3 (b_1 - 1) -  (c - 1)  \, (b_1 - 1) \mathcal{F}^{(1)}_3 (c - 1) = 0;\\
  & 	c \, (b_1 + \theta - 1) \, \mathcal{F}^{(1)}_3 (b_1 - 1) - (b_1 - 1)  (c + \theta +\phi)  \,  \mathcal{F}^{(1)}_3 (c + 1) = 0;\\
  & (b_2 + \phi - 1) \,	(c + \theta + \phi - 1) \,  \mathcal{F}^{(1)}_3 (b_2 - 1) -  (c - 1)  \, (b_2 - 1) \mathcal{F}^{(1)}_3 (c - 1) = 0;\\
  & 	c \, (b_2 + \phi - 1) \, \mathcal{F}^{(1)}_3 (b_2 - 1) - (b_2 - 1)  (c + \theta +\phi)  \,  \mathcal{F}^{(1)}_3 (c + 1) = 0;\\
   & 	c \, (c - 1) \, \mathcal{F}^{(1)}_3 (c - 1) - (c + \theta + \phi - 1)  (c + \theta +\phi)  \,  \mathcal{F}^{(1)}_3 (c + 1) = 0.
 \end{align}
The particular cases of these formulas lead to the differential recursion formulas for Appell function $F_3$ and Kamp\'e de F\'eriet  hypergeometric functions. Now, using  the difference relations, 
\begin{align}
	&	a_1 \, \mathcal{F}^{(1)}_3 (a_1 + 1)  = \left(a_1 + \frac{1}{k_1}\Theta_{t_1} \right) \, \mathcal{F}^{(1)}_3;\\
	& \left(a_1 + \frac{1}{k_1}\Theta_{t_1} - 1\right) \, \mathcal{F}^{(1)}_3 (a_1 - 1)  =	(a_1 - 1) \, \mathcal{F}^{(1)}_3;\\
		&	a_2 \, \mathcal{F}^{(1)}_3 (a_2 + 1)  = \left(a_2 + \frac{1}{k_2}\Theta_{t_2} \right) \, \mathcal{F}^{(1)}_3;\\
	& \left(a_2 + \frac{1}{k_2}\Theta_{t_2} - 1\right) \, \mathcal{F}^{(1)}_3 (a_2 - 1)  =	(a_2 - 1) \, \mathcal{F}^{(1)}_3;\\
	&	b_1 \, \mathcal{F}^{(1)}_3 (b_1 + 1)  = \left(b_1 + \frac{1}{k_1} \Theta_{t_1}\right) \, \mathcal{F}^{(1)}_3;\\
	&\left(b_1 + \frac{1}{k_1} \Theta_{t_1} - 1\right) \, \mathcal{F}^{(1)}_3 (b_1 - 1)  =	(b_1 - 1) \, \mathcal{F}^{(1)}_3;\\
	&	b_2 \, \mathcal{F}^{(1)}_3 (b_2 + 1)  = \left(b_2 + \frac{1}{k_2} \, \Theta_{t_2}\right) \, \mathcal{F}^{(1)}_3;\\
	& \left(b_2 + \frac{1}{k_2} \, \Theta_{t_2} - 1\right) \, \mathcal{F}^{(1)}_3 (b_2 - 1)  =	(b_2 - 1) \, \mathcal{F}^{(1)}_3;\\
	&	(c - 1) \, \mathcal{F}^{(1)}_3 (c - 1)  = \left(c + \frac{1}{k_1} \Theta_{t_1} + \frac{1}{k_2} \Theta_{t_2} - 1\right) \, \mathcal{F}^{(1)}_3;\\
	& \left(c + \frac{1}{k_1} \Theta_{t_1} + \frac{1}{k_2} \Theta_{t_2}\right) \, \mathcal{F}^{(1)}_3 (c + 1)  =	c \, \mathcal{F}^{(1)}_3,
\end{align}
one will get the following 45 difference recursion relations obeyed by discrete Appell function $\mathcal{F}^{(1)}_3$. 
\begin{align}
	& a_1 \, (a_1 - 1) \, \mathcal{F}^{(1)}_3 (a_1 + 1)\nonumber\\
	& \quad - \left(a_1 + \frac{1}{k_1} \Theta_{t_1}\right) \, \left(a_1 + \frac{1}{k_1} \Theta_{t_1}  - 1\right) \mathcal{F}^{(1)}_3 (a_1 - 1) = 0;\\
	& a_1 \, (b_1 - 1) \, \mathcal{F}^{(1)}_3 (a_1 + 1) \nonumber\\
	& \quad - \left(a_1 + \frac{1}{k_1} \Theta_{t_1} \right) \, \left(b_1 + \frac{1}{k_1} \Theta_{t_1}  - 1\right) \mathcal{F}^{(1)}_3 (b_1 - 1) = 0;\\
	& a_1 \, (b_2 - 1) \, \mathcal{F}^{(1)}_3 (a_1 + 1)\nonumber\\
	& \quad - \left(a_1 + \frac{1}{k_1} \Theta_{t_1}\right) \, \left(b_2 + \frac{1}{k_2} \Theta_{t_2} - 1\right) \mathcal{F}^{(1)}_3 (b_2 - 1) = 0;\\
	& a_1 \, c \, \mathcal{F}^{(1)}_3 (a_1 + 1)\nonumber\\
&	\quad  - \left(a_1 + \frac{1}{k_1} \Theta_{t_1}\right) \, \left(c + \frac{1}{k_1} \Theta_{t_1} + \frac{1}{k_2} \Theta_{t_2} \right) \mathcal{F}^{(1)}_3 (c + 1) = 0;\\
	& a_1 \, \left(b_1 + \frac{1}{k_1} \Theta_{t_1}\right) \, \mathcal{F}^{(1)}_3 (a_1 + 1) - b_1 \, \left(a_1 + \frac{1}{k_1} \Theta_{t_1}\right) \, \mathcal{F}^{(1)}_3 (b_1 + 1) = 0;\\
	& a_1 \, \left(b_2 + \frac{1}{k_2} \Theta_{t_2}\right) \, \mathcal{F}^{(1)}_3 (a_1 + 1) - b_2 \, \left(a_1 + \frac{1}{k_1} \Theta_{t_1}\right) \, \mathcal{F}^{(1)}_3 (b_2 + 1) = 0;\\
	& a_1 \, \left(c + \frac{1}{k_1} \Theta_{t_1} + \frac{1}{k_2} \Theta_{t_2} - 1\right) \, \mathcal{F}^{(1)}_3 (a_1 + 1)\nonumber\\
	& \quad  - (c - 1) \, \left(a_1 + \frac{1}{k_1} \Theta_{t_1}\right) \, \mathcal{F}^{(1)}_3 (c - 1) = 0;\\
	&\left(a_1 + \frac{1}{k_1} \Theta_{t_1} - 1\right) \, \left(b_1 + \frac{1}{k_1} \Theta_{t_1} \right) \, \mathcal{F}^{(1)}_3 (a_1 - 1) - b_1 \, (a_1 - 1)  \, \mathcal{F}^{(1)}_3 (b_1 + 1) = 0;\\
	&	\left(a_1 + \frac{1}{k_1} \Theta_{t_1} - 1\right) \, \left(b_2 + \frac{1}{k_2}  \Theta_{t_2}\right) \, \mathcal{F}^{(1)}_3 (a_1 - 1) - b_2 \, (a_1 - 1)  \, \mathcal{F}^{(1)}_3 (b_2 + 1) = 0;\\
	&	\left(a_1 + \frac{1}{k_1} \Theta_{t_1} - 1\right) \, \left(c + \frac{1}{k_1} \Theta_{t_1} + \frac{1}{k_2} \Theta_{t_2} - 1\right) \, \mathcal{F}^{(1)}_3 (a_1 - 1)\nonumber\\
	& \quad  - (c - 1) \, (a_1 - 1)  \, \mathcal{F}^{(1)}_3 (c - 1) = 0;\\
	& (b_1 - 1) \,	\left(a_1 + \frac{1}{k_1} \Theta_{t_1} - 1\right) \,  \mathcal{F}^{(1)}_3 (a_1 - 1) \nonumber\\
	& \quad -  (a_1 - 1)  \, \left(b_1 + \frac{1}{k_1} \Theta_{t_1}  - 1\right) \mathcal{F}^{(1)}_3 (b_1 - 1) = 0;\\
	& (b_2 - 1) \,	\left(a_1 + \frac{1}{k_1} \Theta_{t_1} - 1\right) \,  \mathcal{F}^{(1)}_3 (a_1 - 1)\nonumber\\
	& \quad  -  (a_1 - 1)  \, \left(b_2 + \frac{1}{k_2} \Theta_{t_2} - 1\right) \mathcal{F}^{(1)}_3 (b_2 - 1) = 0;\\
	& c \,	\left(a_1 + \frac{1}{k_1} \Theta_{t_1} - 1\right) \,  \mathcal{F}^{(1)}_3 (a_1 - 1)\nonumber\\
	& \quad -  (a_1 - 1)  \, \left(c + \frac{1}{k_1} \Theta_{t_1} + \frac{1}{k_2} \Theta_{t_2}\right) \mathcal{F}^{(1)}_3 (c + 1) = 0;\\
	&	a_1 \left(a_2 + \frac{1}{k_2} \Theta_{t_2}\right) \,  \mathcal{F}^{(1)}_3 (a_1 + 1)\nonumber\\
	& \quad  - a_2 \, \left(a_1 + \frac{1}{k_1} \Theta_{t_1}\right)  \, \mathcal{F}^{(1)}_3 (a_2 + 1) = 0;\\
	& a_1 (a_2 - 1) \,  \mathcal{F}^{(1)}_3 (a_1 + 1) \nonumber\\
	& \quad -  \left(a_1 + \frac{1}{k_1} \Theta_{t_1}\right)  \, \left(a_2 + \frac{1}{k_2} \Theta_{t_2}  - 1\right) \mathcal{F}^{(1)}_3 (a_2 - 1) = 0;\\
	& \left(a_1 + \frac{1}{k_1} \Theta_{t_1} - 1\right)  \, \left(a_2 + \frac{1}{k_2} \Theta_{t_2}\right) \,  \mathcal{F}^{(1)}_3 (a_1 - 1)\nonumber\\
	& \quad  -  a_2  \, (a_1 - 1) \mathcal{F}^{(1)}_3 (a_2 + 1) = 0;\\
	& a_2 \,	\left(a_1 + \frac{1}{k_1} \Theta_{t_1} - 1\right) \,  \mathcal{F}^{(1)}_3 (a_1 - 1)\nonumber\\
	& \quad -  (a_1 - 1)  \, \left(a_2 + \frac{1}{k_2} \Theta_{t_2} - 1\right) \mathcal{F}^{(1)}_3 (a_2 - 1) = 0;\\	
		& a_2 \, (a_2 - 1) \, \mathcal{F}^{(1)}_3 (a_2 + 1)\nonumber\\
	& \quad - \left(a_2 + \frac{1}{k_2} \Theta_{t_2}\right) \, \left(a_2 + \frac{1}{k_2} \Theta_{t_2}  - 1\right) \mathcal{F}^{(1)}_3 (a_2 - 1) = 0;\\
	& a_2 \, (b_1 - 1) \, \mathcal{F}^{(1)}_3 (a_2 + 1) \nonumber\\
	& \quad - \left(a_2 + \frac{1}{k_2} \Theta_{t_2} \right) \, \left(b_1 + \frac{1}{k_1} \Theta_{t_1}  - 1\right) \mathcal{F}^{(1)}_3 (b_1 - 1) = 0;\\
	& a_2 \, (b_2 - 1) \, \mathcal{F}^{(1)}_3 (a_2 + 1)\nonumber\\
	& \quad - \left(a_2 + \frac{1}{k_2} \Theta_{t_2}\right) \, \left(b_2 + \frac{1}{k_2} \Theta_{t_2} - 1\right) \mathcal{F}^{(1)}_3 (b_2 - 1) = 0;\\
	& a_2 \, c \, \mathcal{F}^{(1)}_3 (a_2 + 1)\nonumber\\
	&	\quad  - \left(a_2 + \frac{1}{k_2} \Theta_{t_2}\right) \, \left(c + \frac{1}{k_1} \Theta_{t_1} + \frac{1}{k_2} \Theta_{t_2} \right) \mathcal{F}^{(1)}_3 (c + 1) = 0;\\
	& a_2 \, \left(b_1 + \frac{1}{k_1} \Theta_{t_1}\right) \, \mathcal{F}^{(1)}_3 (a_2 + 1) - b_1 \, \left(a_2 + \frac{1}{k_2} \Theta_{t_2}\right) \, \mathcal{F}^{(1)}_3 (b_1 + 1) = 0;\\
	& a_2 \, \left(b_2 + \frac{1}{k_2} \Theta_{t_2}\right) \, \mathcal{F}^{(1)}_3 (a_2 + 1) - b_2 \, \left(a_2 + \frac{1}{k_2} \Theta_{t_2}\right) \, \mathcal{F}^{(1)}_3 (b_2 + 1) = 0;\\
	& a_2 \, \left(c + \frac{1}{k_1} \Theta_{t_1} + \frac{1}{k_2} \Theta_{t_2} - 1\right) \, \mathcal{F}^{(1)}_3 (a_2 + 1)\nonumber\\
	& \quad  - (c - 1) \, \left(a_2 + \frac{1}{k_2} \Theta_{t_2}\right) \, \mathcal{F}^{(1)}_3 (c - 1) = 0;\\
	&\left(a_2 + \frac{1}{k_2} \Theta_{t_2} - 1\right) \, \left(b_1 + \frac{1}{k_1} \Theta_{t_1} \right) \, \mathcal{F}^{(1)}_3 (a_2 - 1) - b_1 \, (a_2 - 1)  \, \mathcal{F}^{(1)}_3 (b_1 + 1) = 0;\\
	&	\left(a_2 + \frac{1}{k_2} \Theta_{t_2} - 1\right) \, \left(b_2 + \frac{1}{k_2}  \Theta_{t_2}\right) \, \mathcal{F}^{(1)}_3 (a_2 - 1) - b_2 \, (a_2 - 1)  \, \mathcal{F}^{(1)}_3 (b_2 + 1) = 0;\\
	&	\left(a_2 + \frac{1}{k_2} \Theta_{t_2} - 1\right) \, \left(c + \frac{1}{k_1} \Theta_{t_1} + \frac{1}{k_2} \Theta_{t_2} - 1\right) \, \mathcal{F}^{(1)}_3 (a_2 - 1)\nonumber\\
	& \quad  - (c - 1) \, (a_2 - 1)  \, \mathcal{F}^{(1)}_3 (c - 1) = 0;\\
	& (b_1 - 1) \,	\left(a_2 + \frac{1}{k_2} \Theta_{t_2} - 1\right) \,  \mathcal{F}^{(1)}_3 (a_2 - 1) \nonumber\\
	& \quad -  (a_2 - 1)  \, \left(b_1 + \frac{1}{k_1} \Theta_{t_1}  - 1\right) \mathcal{F}^{(1)}_3 (b_1 - 1) = 0;\\
	& (b_2 - 1) \,	\left(a_2 + \frac{1}{k_2} \Theta_{t_2} - 1\right) \,  \mathcal{F}^{(1)}_3 (a_2 - 1)\nonumber\\
	& \quad  -  (a_2 - 1)  \, \left(b_2 + \frac{1}{k_2} \Theta_{t_2} - 1\right) \mathcal{F}^{(1)}_3 (b_2 - 1) = 0;\\
	& c \,	\left(a_2 + \frac{1}{k_2} \Theta_{t_2} - 1\right) \,  \mathcal{F}^{(1)}_3 (a_2 - 1)\nonumber\\
	& \quad -  (a_2 - 1)  \, \left(c + \frac{1}{k_1} \Theta_{t_1} + \frac{1}{k_2} \Theta_{t_2}\right) \mathcal{F}^{(1)}_3 (c + 1) = 0;\\
	& b_1 \,	(b_1 - 1) \,  \mathcal{F}^{(1)}_3 (b_1 + 1) -  \left(b_1 + \frac{1}{k_1} \Theta_{t_1} \right)  \, \left(b_1 + \frac{1}{k_1} \Theta_{t_1}  - 1\right) \mathcal{F}^{(1)}_3 (b_1 - 1) = 0;\\
	& b_1 \,	\left(b_2 + \frac{1}{k_2} \Theta_{t_2}\right) \,  \mathcal{F}^{(1)}_3 (b_1 + 1) -  b_2  \, \left(b_1 + \frac{1}{k_1} \Theta_{t_1}\right) \mathcal{F}^{(1)}_3 (b_2 + 1) = 0;\\
	& b_1 \,	(b_2 - 1) \,  \mathcal{F}^{(1)}_3 (b_1 + 1) -  \left(b_1 + \frac{1}{k_1} \Theta_{t_1}\right)  \, \left(b_2 + \frac{1}{k_2} \Theta_{t_2} - 1\right) \mathcal{F}^{(1)}_3 (b_2 - 1) = 0;\\
	& b_1 \,	\left(c + \frac{1}{k_1} \Theta_{t_1} + \frac{1}{k_2} \Theta_{t_2} -  1\right) \,  \mathcal{F}^{(1)}_3 (b_1 + 1) -  (c - 1)  \, \left(b_1 + \frac{1}{k_1} \Theta_{t_1} \right) \mathcal{F}^{(1)}_3 (c - 1) = 0;\\
	& b_1 \,	c \,  \mathcal{F}^{(1)}_3 (b_1 + 1) -  \left(c + \frac{1}{k_1} \Theta_{t_1} + \frac{1}{k_2} \Theta_{t_2}\right)  \, \left(b_1 + \frac{1}{k_1} \Theta_{t_1}\right) \mathcal{F}^{(1)}_3 (c + 1) = 0;\\
	& b_2 \,	(b_1 - 1) \,  \mathcal{F}^{(1)}_3 (b_2 + 1) -  \left(b_2 + \frac{1}{k_2} \Theta_{t_2}\right)  \, \left(b_1 + \frac{1}{k_1} \Theta_{t_1}  - 1\right) \mathcal{F}^{(1)}_3 (b_1 - 1) = 0;\\
	& b_2 \,	(b_2 - 1) \,  \mathcal{F}^{(1)}_3 (b_2 + 1) -  \left(b_2 + \frac{1}{k_2} \Theta_{t_2}\right)  \, \left(b_2 + \frac{1}{k_2} \Theta_{t_2}  - 1\right) \mathcal{F}^{(1)}_3 (b_2 - 1) = 0;\\
	& b_2 \,	\left(c + \frac{1}{k_1} \Theta_{t_1} + \frac{1}{k_2} \Theta_{t_2} - 1\right) \,  \mathcal{F}^{(1)}_3 (b_2 + 1) -  (c - 1)  \, \left(b_2 + \frac{1}{k_2}  \Theta_{t_2}\right) \mathcal{F}^{(1)}_3 (c - 1) = 0;\\
	& b_2 \,	c \,  \mathcal{F}^{(1)}_3 (b_2 + 1) -  \left(c + \frac{1}{k_1} \Theta_{t_1} + \frac{1}{k_2} \Theta_{t_2}\right)  \, \left(b_2 + \frac{1}{k_2} \Theta_{t_2}\right) \mathcal{F}^{(1)}_3 (c + 1) = 0;\\
	& (b_2 - 1) \,	\left(b_1 + \frac{1}{k_1} \Theta_{t_1}  - 1\right) \,  \mathcal{F}^{(1)}_3 (b_1 - 1) -  (b_1 - 1)  \, \left(b_2 + \frac{1}{k_2}  \Theta_{t_2} - 1\right) \mathcal{F}^{(1)}_3 (b_2 - 1) = 0;\\
	& \left(b_1 + \frac{1}{k_1} \Theta_{t_1} - 1\right) \,	\left(c + \frac{1}{k_1} \Theta_{t_1} + \frac{1}{k_2} \Theta_{t_2} - 1\right) \,  \mathcal{F}^{(1)}_3 (b_1 - 1)\nonumber\\
	& \quad -  (c - 1)  \, (b_1 - 1) \mathcal{F}^{(1)}_3 (c - 1) = 0;\\
	& 	c \, \left(b_1 + \frac{1}{k_1} \Theta_{t_1} - 1\right) \, \mathcal{F}^{(1)}_3 (b_1 - 1) - (b_1 - 1)  \left(c + \frac{1}{k}(\Theta_{t_1} + \Theta_{t_2})\right)  \,  \mathcal{F}^{(1)}_3 (c + 1) = 0;\\
	& \left(b_2 + \frac{1}{k_2} \Theta_{t_2} - 1\right) \,	\left(c + \frac{1}{k_1} \Theta_{t_1} + \frac{1}{k_2} \Theta_{t_2} - 1\right) \,  \mathcal{F}^{(1)}_3 (b_2 - 1)\nonumber\\
	& \quad -  (c - 1)  \, (b_2 - 1) \mathcal{F}^{(1)}_3 (c - 1) = 0;\\
	& 	c \, \left(b_2 + \frac{1}{k_2}  \Theta_{t_2} - 1\right) \, \mathcal{F}^{(1)}_3 (b_2 - 1) - (b_2 - 1)  \left(c + \frac{1}{k_1} \Theta_{t_1} + \frac{1}{k_2} \Theta_{t_2}\right)  \,  \mathcal{F}^{(1)}_3 (c + 1) = 0;\\
	& 	c \, (c - 1) \, \mathcal{F}^{(1)}_3 (c - 1) - \left(c + \frac{1}{k_1} \Theta_{t_1} + \frac{1}{k_2} \Theta_{t_2} - 1\right)  \left(c + \frac{1}{k_1} \Theta_{t_1} + \frac{1}{k_2} \Theta_{t_2}\right)  \,  \mathcal{F}^{(1)}_3 (c + 1) = 0.
\end{align}

\section{The discrete Appell function $\mathcal{F}_3^{(2)}$}
In this section, we examine the second form of discrete Appell function $\mathcal{F}^{(2)}_3$ given by \eqref{3.2} and find its regions of convergence, difference-differential equations, integral representations, difference-differential formulas, finite and infinite summation formulas. As the method of proofs is same as that for the discrete Appell function $\mathcal{F}^{(1)}_3$, we list the results for  $\mathcal{F}^{(2)}_3$ without proof.

For different values of $k$, the discrete Appell function $\mathcal{F}^{(2)}_3$ reduces into other functions. In particular, for $k = 0$, we have
\begin{align}
	\mathcal{F}^{(2)}_3(a_1, a_2, b_1, b_2; c; t, 0, x, y)   = F_3 (a_1, a_2, b_1, b_2; c; x, y).
\end{align}
For $k = 1$, we get
\begin{align}
	\mathcal{F}^{(2)}_3(a_1, a_2, b_1, b_2; c; t, 1, x, y)
	 = F_{{1}:{0}, {0}} ^{{1}:{2}, {2}}\left(\begin{array}{ccc}
		-t: & a_1, b_1; & a_2,  b_2\\
		c: & - ; & - 
	\end{array}; x, \, y\right).
\end{align}

The limiting cases of discrete Appell function $\mathcal{F}^{(2)}_3$ give the discrete Humbert functions. We denote them by $\varXi^{(2)}_1$ and $\varXi^{(2)}_2$, defined as
\begin{align}
	&\varXi^{(2)}_1 \left(a_1, a_2, b_1; c; t, k, x,  y\right) \nonumber\\
	& = \sum_{m,n \geq 0} \frac{(a_1)_{m} \, (a_2)_n \, (b_1)_m \, (-1)^{(m + n) k} \, (-t)_{(m + n)k}}{ (c)_{m+n} \, m! \, n!} \ x^m \, y^n; 
	\\[5pt]
&	\varXi^{(2)}_2 \left(a_1, b_1; c; t, k, x,  y\right)\nonumber\\
& = \sum_{m,n \geq 0} \frac{(a_1)_{m} \, (b_1)_m \, (-1)^{(m + n) k} \, (-t)_{(m + n)k}}{(c)_{m+n} \, m! \, n!} \ x^m \, y^n.
\end{align}
One can easily verify that
\begin{align}
	& \lim_{\varepsilon \to 0}  \mathcal{F}^{(2)}_3 \left(a_1, a_2, b_1, \frac{1}{\varepsilon}; c; t, k, x, \varepsilon \, y\right)  = \varXi^{(2)}_1 \left(a_1, a_2, b_1; c; t, k, x,  y\right)\nonumber\\
	&\lim_{\varepsilon \to 0}  \mathcal{F}^{(2)}_3 \left(a_1, \frac{1}{\varepsilon}, b_1,\frac{1}{\varepsilon}; c; t, k, x, \varepsilon^2 \, y\right)  = \varXi^{(2)}_2 \left(a_1, b_1; c; t, k, x,  y\right).
\end{align}
The following difference-differential equation are satisfied by discrete Appell function $\mathcal{F}^{(2)}_3$
\begin{align}
	\left[\theta \left(\frac{1}{k} \Theta_{t}  + c - 1\right) -   (-1)^k \, (-t)_k \, x \, \rho_t^k \,  \left(a_1 + \theta\right) \, \left(b_1 + \theta\right) \right] \mathcal{F}^{(2)}_3 = 0\nonumber
	\\[5pt]
	\left[\phi \, \left(\frac{1}{k} \Theta_{t} + c - 1\right) -  (-1)^k \, (-t)_k \, y \, \rho_t^k \,  \left(a_2 + \phi\right) \, \left(b_2 + \phi \right) \right] \mathcal{F}^{(2)}_3 = 0.
\end{align}	    

\subsection{Integral representations}
\begin{theorem}
	Let $a_1$, $a_2$, $b_1$, $b_2$, $c$ and $t$ be complex numbers. Then for $\vert x\vert < 1$, $\vert y\vert < 1$, the discrete Appell function $\mathcal{F}^{(2)}_3$ can be represented in the integral forms as
	\begin{align}
		& \mathcal{F}^{(2)}_3(a_1, a_2, b_1, b_2; c; t, k, x, y)\nonumber\\
		& =\Gamma \left(\begin{array}{c}
			c\\
			b_1, b_2, c - b_1 - b_2
		\end{array}\right) \iint u^{b_1 - 1} v^{b_2 - 1} (1-u-v)^{c -  b_1 - b_2 - 1} \nonumber\\
		& \quad \times F_{{0}:{0}; {0}} ^{{k}:{ 1}; {2}}\left(\begin{array}{ccc}
			\frac{- t}{k}, \dots, \frac{- t + k - 1}{k} : & a_1 ; & a_2\\
			-: & - ; & - 
		\end{array}; (-k)^k \, u \, x,  (-k)^k \, v \, y\right) \, du \, dv,\nonumber\\
	& \qquad \qquad \qquad	u\geq 0, \, v\geq 0, \, 1-u-v\geq 0;\nonumber\\
		& = \frac{1}{\Gamma (a_1)} \int_{0}^{\infty} e^{-u} \, u^{a_1 - 1} \nonumber\\
		& \quad \times F_{{1}:{0}; {0}} ^{{k}:{ 1}; {2}}\left(\begin{array}{ccc}
			\frac{- t}{k}, \dots, \frac{- t + k - 1}{k} : & b_1 ; & a_2,  b_2\\
			c: & - ; & - 
		\end{array}; (-k)^k \, u x,  (-k)^k \, y\right) du;\nonumber\\
		& = \frac{1}{\Gamma (a_2)} \int_{0}^{\infty} e^{-v} \, v^{a_2 - 1} \nonumber\\
		& \quad \times F_{{1}:{0}; {0}} ^{{k}:{ 2}; { 1}}\left(\begin{array}{ccc}
		\frac{- t}{k}, \dots, \frac{- t + k - 1}{k}: & a_1, b_1; &  b_2\\
			c: & - ; & - 
		\end{array}; (-k)^k \,  x,  (-k)^k \, v y\right) dv;\nonumber\\
		& = \frac{1}{\Gamma (b_1)} \int_{0}^{\infty} e^{-u} \, u^{b_1 - 1} \nonumber\\
		& \quad \times F_{{1}:{0}; {0}} ^{{k}:{1}; {2}}\left(\begin{array}{ccc}
			\frac{- t}{k}, \dots, \frac{- t + k - 1}{k}: & a_1; & a_2, b_2\\
			c: & - ; & - 
		\end{array}; (-k)^k \, u x,  (-k)^k \, y\right) du;\nonumber\\
		& = \frac{1}{\Gamma (b_2)} \int_{0}^{\infty} e^{-v} \, v^{b_2 - 1} \nonumber\\
		& \quad \times F_{{1}:{0}; {0}} ^{{k}:{2}; { 1}}\left(\begin{array}{ccc}
			\frac{- t}{k}, \dots, \frac{- t + k - 1}{k}: & a_1, b_1; & a_2\\
			c: & - ; & - 
		\end{array}; (-k)^k \, x,  (-k)^k \, v y\right) dv.
	\end{align}
\end{theorem}	
\subsection{Differential formulae}
The following  differential formulae are satisfied by discrete Appell function $\mathcal{F}^{(2)}_3$:
\begin{align}
	& (\theta)^r \mathcal{F}^{(2)}_3(a_1, a_2, b_1, b_2; c; t, k, x, y) \nonumber\\
	& = \frac{(-1)^{rk} \, (a_1)_r \, (b_1)_r \, (-t)_{rk} \, x^r}{(c)_r} \, \mathcal{F}^{(2)}_3(a_1 + r, a_2, b_1 + r, b_2; c + r; t - rk, k, x, y);\\
	& (\phi)^r \mathcal{F}^{(2)}_3(a_1, a_2, b_1, b_2; c; t, k, x, y) \nonumber\\
	& = \frac{(-1)^{rk} \, (a_2)_r \, (b_2)_r \, (-t)_{rk} \, y^r}{(c)_r} \, \mathcal{F}^{(2)}_3(a_1, a_2 + r, b_1, b_2 + r; c + r; t - rk, k, x, y);\nonumber\\
		& \left(\frac{\partial}{\partial x}\right)^r \left[x^{b_1 + r - 1} \mathcal{F}^{(2)}_3(a_1, a_2, b_1, b_2; c; t, k, x, y)\right]\nonumber\\
	& = x^{b_1 - 1} \, (b_1)_r \mathcal{F}^{(2)}_3(a_1, a_2, b_1 + r, b_2; c; t, k, x, y);\\
	& \left(\frac{\partial}{\partial y}\right)^r [y^{b_2 + r - 1} \mathcal{F}^{(2)}_3(a_1, a_2, b_1, b_2; c; t, k, x, y)]\nonumber\\
	& = y^{b_2 - 1} \, (b_2)_r \mathcal{F}^{(2)}_3(a_1, a_2, b_1, b_2 + r; c; t, k, x, y);\\
	& \left(\frac{\partial}{\partial x}\right)^r [x^{a_1 + r - 1} \mathcal{F}^{(2)}_3(a_1, a_2, b_1, b_2; c; t, k, x, y)]\nonumber\\
	& = x^{a_1 - 1} \, (a_1)_r \mathcal{F}^{(2)}_3(a_1 + r, a_2, b_1, b_2; c; t, k, x, y);\\
	& \left(\frac{\partial}{\partial y}\right)^r [y^{a_2 + r - 1} \mathcal{F}^{(2)}_3(a_1, a_2, b_1, b_2; c; t, k, x, y)]\nonumber\\
	& = y^{a_2 - 1} \, (a_2)_r \mathcal{F}^{(2)}_3(a_1, a_2 + r, b_1, b_2; c; t, k, x, y);\\
	& \left(\frac{\partial}{\partial x}\right)^r [x^{c - 1} \mathcal{F}^{(2)}_3(a_1, a_2, b_1, b_2; c; t, k, x, xy)]\nonumber\\
	& = (-1)^r \, (1 - c)_r \, x^{c - r - 1} \mathcal{F}^{(2)}_3(a_1, a_2, b_1, b_2; c - r; t, k, x, xy);\\
	& \left(\frac{\partial}{\partial y}\right)^r [y^{c - 1} \mathcal{F}^{(2)}_3(a_1, a_2, b_1, b_2; c; t, k, xy, y)]\nonumber\\
	& = (-1)^r \, y^{c - r - 1} \, (1 - c)_r \mathcal{F}^{(2)}_3(a_1, a_2, b_1, b_2; c - r; t, k, xy, y).
\end{align}
\subsection{Finite and infinite sums}
Next, we state some finite and infinite sums in terms of discrete Appell function $\mathcal{F}^{(2)}_3$ in the following theorem.
\begin{theorem}
	\begin{align}
		& \mathcal{F}^{(2)}_3(a_1, a_2, b_1 + r, b_2; c; t, k, x, y)\nonumber\\
		& = \sum_{s = 0}^{r} {r \choose s} \frac{(a_1)_s \, (-1)^{sk} \, (-t)_{sk}}{(c)_s} \, x^s \, \mathcal{F}^{(2)}_3(a_1 + s, a_2, b_1 + s, b_2; c + s; t - sk, k, x, y);\\
		& \mathcal{F}^{(2)}_3(a_1, a_2, b_1, b_2 + r; c; t, k, x, y)\nonumber\\
		& = \sum_{s = 0}^{r} {r \choose s} \frac{(a_2)_s \, (-1)^{sk} \, (-t)_{sk}}{(c)_s} \, y^s \, \mathcal{F}^{(2)}_3(a_1, a_2 + s, b_1, b_2 + s; c + s; t - sk, k, x, y).
	\end{align}
\end{theorem}  
We also have  an infinite summation formula, that can be regarded as a transformation formula, in terms of discrete Appell function $\mathcal{F}^{(2)}_3$:
\begin{align}
	&\sum_{r = 0}^{\infty} \frac{(a_1)_r}{r !} \, z^r \, \mathcal{F}^{(2)}_3 (a_1 + r, a_2, b_1, b_2; c ; t, k, x, y)\nonumber\\
	& = (1 - z)^{-a_1} \, \mathcal{F}^{(2)}_3 \left(a_1, a_2, b_1, b_2; c; t, k, \frac{x}{1 - z}, y\right). 
\end{align}
\subsection{Recursion Formulae}
We have the following recursion formulae
\begin{align}
	& \mathcal{F}^{(2)}_3 (a_1 + s, a_2, b_1, b_2; c; t, k, x, y) \nonumber\\
	& = \mathcal{F}^{(2)}_3 (a_1, a_2, b_1, b_2; c ; t, k, x, y)\nonumber\\
	& \quad + \frac{(-1)^k \, (-t)_k \, b_1 \, x}{c}  \sum_{r = 1}^{s} \mathcal{F}^{(2)}_3 (a_1 + r, a_2, b_1 + 1, b_2; c + 1; t - k, k, x, y);\\
	& \mathcal{F}^{(2)}_3 (a_1 - s, a_2, b_1, b_2; c ; t, k, x, y) \nonumber\\
	& = \mathcal{F}^{(2)}_3 (a_1, a_2, b_1, b_2; c ; t, k, x, y)\nonumber\\
	& \quad - \frac{(-1)^k \, (-t)_k \, b_1 \, x}{c}  \sum_{r = 0}^{s - 1} \mathcal{F}^{(2)}_3 (a_1 - r, a_2, b_1 + 1, b_2; c + 1; t - k, k, x, y);\\
	& \mathcal{F}^{(2)}_3 (a_1, a_2 + s, b_1, b_2; c; t, k, x, y) \nonumber\\
	& = \mathcal{F}^{(2)}_3 (a_1, a_2, b_1, b_2; c ; t, k, x, y)\nonumber\\
	& \quad + \frac{(-1)^k \, (-t)_k \, b_2 \, y}{c} \sum_{r = 1}^{s} \mathcal{F}^{(2)}_3 (a_1, a_2 + r, b_1, b_2 + 1; c + 1; t - k, k, x, y);\\
	& \mathcal{F}^{(2)}_3 (a_1, a_2 - s, b_1, b_2; c ; t, k, x, y) \nonumber\\
	& = \mathcal{F}^{(2)}_3 (a_1, a_2, b_1, b_2; c ; t, k, x, y)\nonumber\\
	& \quad - \frac{(-1)^k \, (-t)_k \, b_2 \, y}{c}  \sum_{r = 0}^{s - 1} \mathcal{F}^{(2)}_3 (a_1, a_2 - r, b_1, b_2 + 1; c + 1; t - k, k, x, y);\\
	& \mathcal{F}^{(2)}_3 (a_1, a_2, b_1 + s, b_2; c ; t, k, x, y) \nonumber\\
	& = \mathcal{F}^{(2)}_3 (a_1, a_2, b_1, b_2; c ; t, k, x, y)\nonumber\\
	& \quad + \frac{(-1)^k \, (-t)_k \, a \, x}{c}  \sum_{r = 1}^{s} \mathcal{F}^{(2)}_3 (a_1 + 1, a_2, b_1 + r, b_2; c + 1; t - k, k, x, y);\\
	& \mathcal{F}^{(2)}_3 (a_1, a_2, b_1 - s, b_2; c ; t, k, x, y) \nonumber\\
	& = \mathcal{F}^{(2)}_3 (a_1, a_2, b_1, b_2; c ; t, k, x, y)\nonumber\\
	& \quad - \frac{(-1)^k \, (-t)_k \, a_1 \, x}{c}  \sum_{r = 0}^{s - 1} \mathcal{F}^{(2)}_3 (a_1 + 1, b_1 - r, b_2; c + 1; t - k, k, x, y);\\
	& \mathcal{F}^{(2)}_3 (a_1, a_2, b_1, b_2; c - s; t, k, x, y) \nonumber\\
	& = \mathcal{F}^{(2)}_3 (a_1, a_2, b_1, b_2; c ; t, k, x, y)\nonumber\\
	& \quad + (-1)^k \, (-t)_k \, a_1 \, b_1 \, x  \sum_{r = 1}^{s} \frac{\mathcal{F}^{(2)}_3 (a_1 + 1, a_2, b_1 + 1, b_2; c + 2 - r; t - k, k, x, y)}{(c - r) \, (c - r + 1)}\nonumber\\
	& \quad + (-1)^k \, (-t)_k \, a_2 \, b_2 \, y  \sum_{r = 1}^{s} \frac{\mathcal{F}^{(2)}_3 (a_1, a_2 + 1, b_1, b_2 + 1; c + 2 - r; t - k, k, x, y)}{(c - r) \, (c - r + 1)}.
\end{align}
Besides, the following relations can be easily verified:
\begin{align}
	&	a_1 \, \mathcal{F}^{(2)}_3 (a_1 + 1)  = (a_1 + \theta) \, \mathcal{F}^{(2)}_3;\\
	& (a_1 + \theta - 1) \, \mathcal{F}^{(2)}_3 (a_1 - 1)  =	(a_1 - 1) \, \mathcal{F}^{(2)}_3;\\
	&	a_2 \, \mathcal{F}^{(2)}_3 (a_2 + 1)  = (a_2 + \phi) \, \mathcal{F}^{(2)}_3;\\
	& (a_2 + \phi - 1) \, \mathcal{F}^{(2)}_3 (a_2 - 1)  =	(a_2 - 1) \, \mathcal{F}^{(2)}_3;\\
	&	b_1 \, \mathcal{F}^{(2)}_3 (b_1 + 1)  = (b_1 + \theta) \, \mathcal{F}^{(2)}_3;\\
	&(b_1 + \theta - 1) \, \mathcal{F}^{(2)}_3 (b_1 - 1)  =	(b_1 - 1) \, \mathcal{F}^{(2)}_3;\\
	&	b_2 \, \mathcal{F}^{(2)}_3 (b_2 + 1)  = (b_2 + \phi) \, \mathcal{F}^{(2)}_3;\\
	& (b_2  + \phi - 1) \, \mathcal{F}^{(2)}_3 (b_2 - 1)  =	(b_2 - 1) \, \mathcal{F}^{(2)}_3;\\
	&	(c - 1) \, \mathcal{F}^{(1)}_3 (c - 1)  = (c + \theta + \phi - 1) \, \mathcal{F}^{(1)}_3;\\
	& (c + \theta + \phi) \, \mathcal{F}^{(1)}_3 (c + 1)  =	c \, \mathcal{F}^{(1)}_3.
\end{align}
On combining any two of the above relations, we get a first order or second order differential recursion relations as:
\begin{align}
	& a_1 \, (a_1 - 1) \, \mathcal{F}^{(2)}_3 (a_1 + 1) - (a_1 + \theta) \, (a_1 + \theta - 1) \mathcal{F}^{(2)}_3 (a_1 - 1) = 0;\\
	& a_1 \, (b_1 - 1) \, \mathcal{F}^{(2)}_3 (a_1 + 1) - (a_1 + \theta) \, (b_1 + \theta  - 1) \mathcal{F}^{(2)}_3 (b_1 - 1) = 0;\\
	& a_1 \, (b_2 - 1) \, \mathcal{F}^{(2)}_3 (a_1 + 1) - (a_1 + \theta) \, (b_2 + \phi - 1) \mathcal{F}^{(2)}_3 (b_2 - 1) = 0;\\
	& a_1 \, c \, \mathcal{F}^{(2)}_3 (a_1 + 1) - (a_1 + \theta) \, (c + \theta + \phi) \mathcal{F}^{(2)}_3 (c + 1) = 0;\\
	& a_1 \, (b_1 + \theta) \, \mathcal{F}^{(2)}_3 (a_1 + 1) - b_1 \, (a_1 + \theta) \, \mathcal{F}^{(2)}_3 (b_1 + 1) = 0;\\
	& a_1 \, (b_2 + \phi) \, \mathcal{F}^{(2)}_3 (a_1 + 1) - b_2 \, (a_1 + \theta) \, \mathcal{F}^{(2)}_3 (b_2 + 1) = 0;\\
	& a_1 \, (c + \theta + \phi - 1) \, \mathcal{F}^{(2)}_3 (a_1 + 1) - (c - 1) \, (a_1 + \theta) \, \mathcal{F}^{(2)}_3 (c - 1) = 0;\\	
	& a_1 \, (a_2 + \phi) \, \mathcal{F}^{(2)}_3 (a_1 + 1) - a_2 \, (a_1 + \theta) \, \mathcal{F}^{(2)}_3 (a_2 + 1) = 0;\\
	& a_1 \, (a_2 - 1) \, \mathcal{F}^{(2)}_3 (a_1 + 1) - (a_1 + \theta) \, (a_2 + \phi - 1) \, \mathcal{F}^{(2)}_3 (a_2 - 1) = 0;\\
	& (a_1 + \theta - 1) \, (a_2 + \phi) \, \mathcal{F}^{(2)}_3 (a_1 - 1) - a_2 \, (a_1 - 1) \mathcal{F}^{(2)}_3 (a_2 + 1) = 0;\\
	& (a_2 - 1) \, (a_1 + \theta - 1) \, \mathcal{F}^{(2)}_3 (a_1 - 1) - (a_1 - 1) \, (a_2 + \phi  - 1) \mathcal{F}^{(2)}_3 (a_2 - 1) = 0;\\
	& a_2 \, (a_2 - 1) \, \mathcal{F}^{(2)}_3 (a_2 + 1) - (a_2 + \phi) \, (a_2 + \phi - 1) \mathcal{F}^{(2)}_3 (a_2 - 1) = 0;\\
	& a_2 \, (b_1 - 1) \, \mathcal{F}^{(2)}_3 (a_2 + 1) - (a_2 + \phi) \, (b_1 + \theta  - 1) \mathcal{F}^{(2)}_3 (b_1 - 1) = 0;\\
	& a_2 \, (b_2 - 1) \, \mathcal{F}^{(2)}_3 (a_2 + 1) - (a_2 + \phi) \, (b_2 + \phi - 1) \mathcal{F}^{(2)}_3 (b_2 - 1) = 0;\\
	& a_2 \, c \, \mathcal{F}^{(2)}_3 (a_2 + 1) - (a_2 + \phi) \, (c + \theta + \phi) \mathcal{F}^{(2)}_3 (c + 1) = 0;\\
	& a_2 \, (b_1 + \theta) \, \mathcal{F}^{(2)}_3 (a_2 + 1) - b_1 \, (a_2 + \phi) \, \mathcal{F}^{(2)}_3 (b_1 + 1) = 0;\\
	& a_2 \, (b_2 + \phi) \, \mathcal{F}^{(2)}_3 (a_2 + 1) - b_2 \, (a_2 + \phi) \, \mathcal{F}^{(2)}_3 (b_2 + 1) = 0;\\
	& a_2 \, (c + \theta + \phi - 1) \, \mathcal{F}^{(2)}_3 (a_2 + 1) - (c - 1) \, (a_2 + \phi) \, \mathcal{F}^{(2)}_3 (c - 1) = 0;\\
	&(a_1 + \theta - 1) \, (b_1 + \theta) \, \mathcal{F}^{(2)}_3 (a_1 - 1) - b_1 \, (a_1 - 1)  \, \mathcal{F}^{(2)}_3 (b_1 + 1) = 0;\\
	&	(a_1 + \theta - 1) \, (b_2 + \phi) \, \mathcal{F}^{(2)}_3 (a_1 - 1) - b_2 \, (a_1 - 1)  \, \mathcal{F}^{(2)}_3 (b_2 + 1) = 0;\\
	&	(a_1 + \theta - 1) \, (c + \theta + \phi - 1) \, \mathcal{F}^{(2)}_3 (a_1 - 1) - (c - 1) \, (a_1 - 1)  \, \mathcal{F}^{(2)}_3 (c - 1) = 0;\\
	& (b_1 - 1) \,	(a_1 + \theta - 1) \,  \mathcal{F}^{(2)}_3 (a_1 - 1) -  (a_1 - 1)  \, (b_1 + \theta - 1) \mathcal{F}^{(2)}_3 (b_1 - 1) = 0;\\
	& (b_2 - 1) \,	(a_1 + \theta - 1) \,  \mathcal{F}^{(2)}_3 (a_1 - 1) -  (a_1 - 1)  \, (b_2 + \phi - 1) \mathcal{F}^{(2)}_3 (b_2 - 1) = 0;\\
	& c \,	(a_1 + \theta - 1) \,  \mathcal{F}^{(2)}_3 (a_1 - 1) -  (a_1 - 1)  \, (c + \theta + \phi) \mathcal{F}^{(2)}_3 (c + 1) = 0;\\
	&(a_2 + \phi - 1) \, (b_1 + \theta) \, \mathcal{F}^{(2)}_3 (a_2 - 1) - b_1 \, (a_2 - 1)  \, \mathcal{F}^{(2)}_3 (b_1 + 1) = 0;\\
	&	(a_2 + \phi - 1) \, (b_2 + \phi) \, \mathcal{F}^{(2)}_3 (a_2 - 1) - b_2 \, (a_2 - 1)  \, \mathcal{F}^{(2)}_3 (b_2 + 1) = 0;\\
	&	(a_2 + \phi - 1) \, (c + \theta + \phi - 1) \, \mathcal{F}^{(2)}_3 (a_2 - 1) - (c - 1) \, (a_2 - 1)  \, \mathcal{F}^{(2)}_3 (c - 1) = 0;\\
	& (b_1 - 1) \,	(a_2 + \phi - 1) \,  \mathcal{F}^{(2)}_3 (a_2 - 1) -  (a_2 - 1)  \, (b_1 + \theta - 1) \mathcal{F}^{(2)}_3 (b_1 - 1) = 0;\\
	& (b_2 - 1) \,	(a_2 + \phi - 1) \,  \mathcal{F}^{(2)}_3 (a_2 - 1) -  (a_2 - 1)  \, (b_2 + \phi - 1) \mathcal{F}^{(2)}_3 (b_2 - 1) = 0;\\
	& c \,	(a_2 + \phi - 1) \,  \mathcal{F}^{(2)}_3 (a_2 - 1) -  (a_2 - 1)  \, (c + \theta + \phi) \mathcal{F}^{(2)}_3 (c + 1) = 0;\\
	& b_1 \,	(b_1 - 1) \,  \mathcal{F}^{(2)}_3 (b_1 + 1) -  (b_1 + \theta)  \, (b_1 + \theta - 1) \mathcal{F}^{(2)}_3 (b_1 - 1) = 0;\\
	& b_1 \,	(b_2 + \phi) \,  \mathcal{F}^{(2)}_3 (b_1 + 1) -  b_2  \, (b_1 + \theta ) \mathcal{F}^{(2)}_3 (b_2 + 1) = 0;\\
	& b_1 \,	(b_2 - 1) \,  \mathcal{F}^{(2)}_3 (b_1 + 1) -  (b_1 + \theta)  \, (b_2 + \phi - 1) \mathcal{F}^{(2)}_3 (b_2 - 1) = 0;\\
	& b_1 \,	(c + \theta + \phi - 1) \,  \mathcal{F}^{(2)}_3 (b_1 + 1) -  (c - 1)  \, (b_1 + \theta) \mathcal{F}^{(2)}_3 (c - 1) = 0;\\
	& b_1 \,	c \,  \mathcal{F}^{(2)}_3 (b_1 + 1) -  (c + \theta +\phi)  \, (b_1 + \theta) \mathcal{F}^{(2)}_3 (c + 1) = 0;\\
	& b_2 \,	(b_1 - 1) \,  \mathcal{F}^{(2)}_3 (b_2 + 1) -  (b_2 + \phi)  \, (b_1 + \theta - 1) \mathcal{F}^{(2)}_3 (b_1 - 1) = 0;\\
	& b_2 \,	(b_2 - 1) \,  \mathcal{F}^{(2)}_3 (b_2 + 1) -  (b_2 + \phi)  \, (b_2 + \phi - 1) \mathcal{F}^{(2)}_3 (b_2 - 1) = 0;\\
	& b_2 \,	(c + \theta + \phi - 1) \,  \mathcal{F}^{(2)}_3 (b_2 + 1) -  (c - 1)  \, (b_2 + \phi) \mathcal{F}^{(2)}_3 (c - 1) = 0;\\
	& b_2 \,	c \,  \mathcal{F}^{(2)}_3 (b_2 + 1) -  (c + \theta +\phi)  \, (b_2 + \phi) \mathcal{F}^{(2)}_3 (c + 1) = 0;\\
	& (b_2 - 1) \,	(b_1 + \theta - 1) \,  \mathcal{F}^{(2)}_3 (b_1 - 1) -  (b_1 - 1)  \, (b_2 + \phi - 1) \mathcal{F}^{(2)}_3 (b_2 - 1) = 0;\\
	& (b_1 + \theta - 1) \,	(c + \theta + \phi - 1) \,  \mathcal{F}^{(2)}_3 (b_1 - 1) -  (c - 1)  \, (b_1 - 1) \mathcal{F}^{(2)}_3 (c - 1) = 0;\\
	& 	c \, (b_1 + \theta - 1) \, \mathcal{F}^{(2)}_3 (b_1 - 1) - (b_1 - 1)  (c + \theta +\phi)  \,  \mathcal{F}^{(2)}_3 (c + 1) = 0;\\
	& (b_2 + \phi - 1) \,	(c + \theta + \phi - 1) \,  \mathcal{F}^{(2)}_3 (b_2 - 1) -  (c - 1)  \, (b_2 - 1) \mathcal{F}^{(2)}_3 (c - 1) = 0;\\
	& 	c \, (b_2 + \phi - 1) \, \mathcal{F}^{(2)}_3 (b_2 - 1) - (b_2 - 1)  (c + \theta +\phi)  \,  \mathcal{F}^{(2)}_3 (c + 1) = 0;\\
	& 	c \, (c - 1) \, \mathcal{F}^{(2)}_3 (c - 1) - (c + \theta + \phi - 1)  (c + \theta +\phi)  \,  \mathcal{F}^{(2)}_3 (c + 1) = 0.
\end{align}
Similarly using  the difference-differential relations, 
\begin{align}
	&	a_1 \, \mathcal{F}^{(2)}_3 (a_1 + 1)  = \left(a_1 + \theta \right) \, \mathcal{F}^{(2)}_3;\\
	& \left(a_1 + \theta - 1\right) \, \mathcal{F}^{(2)}_3 (a_1 - 1)  =	(a_1 - 1) \, \mathcal{F}^{(2)}_3;\\
	&	a_2 \, \mathcal{F}^{(2)}_3 (a_2 + 1)  = \left(a_2 + \phi \right) \, \mathcal{F}^{(2)}_3;\\
	& \left(a_2 + \phi - 1\right) \, \mathcal{F}^{(2)}_3 (a_2 - 1)  =	(a_2 - 1) \, \mathcal{F}^{(2)}_3;\\
	&	b_1 \, \mathcal{F}^{(2)}_3 (b_1 + 1)  = \left(b_1 + \theta\right) \, \mathcal{F}^{(2)}_3;\\
	&\left(b_1 + \theta - 1\right) \, \mathcal{F}^{(2)}_3 (b_1 - 1)  =	(b_1 - 1) \, \mathcal{F}^{(2)}_3;\\
	&	b_2 \, \mathcal{F}^{(2)}_3 (b_2 + 1)  = \left(b_2 + \phi\right) \, \mathcal{F}^{(2)}_3;\\
	& \left(b_2 + \phi - 1\right) \, \mathcal{F}^{(2)}_3 (b_2 - 1)  =	(b_2 - 1) \, \mathcal{F}^{(2)}_3;\\
	&	(c - 1) \, \mathcal{F}^{(2)}_3 (c - 1)  = \left(c + \frac{1}{k} \Theta_{t}  - 1\right) \, \mathcal{F}^{(2)}_3;\\
	& \left(c + \frac{1}{k} \Theta_{t}\right) \, \mathcal{F}^{(2)}_3 (c + 1)  =	c \, \mathcal{F}^{(2)}_3,
\end{align}
one will get the following difference-differential recursion formulas obeyed by discrete Appell function $\mathcal{F}^{(2)}_3$. 
\begin{align}
	& a_1 \, c \, \mathcal{F}^{(2)}_3 (a_1 + 1)  - \left(a_1 + \theta\right) \, \left(c + \frac{1}{k} \Theta_{t}\right) \mathcal{F}^{(2)}_3 (c + 1) = 0;\\
	& a_1 \, \left(c + \frac{1}{k} \Theta_{t} - 1\right) \, \mathcal{F}^{(2)}_3 (a_1 + 1)  - (c - 1) \, \left(a_1 + \theta\right) \, \mathcal{F}^{(2)}_3 (c - 1) = 0;\\
	&	\left(a_1 + \theta - 1\right) \, \left(c + \frac{1}{k} \Theta_{t} - 1\right) \, \mathcal{F}^{(2)}_3 (a_1 - 1)  - (c - 1) \, (a_1 - 1)  \, \mathcal{F}^{(2)}_3 (c - 1) = 0;\\
	& c \,	\left(a_1 + \theta - 1\right) \,  \mathcal{F}^{(2)}_3 (a_1 - 1) -  (a_1 - 1)  \, \left(c + \frac{1}{k} \Theta_{t}\right) \mathcal{F}^{(2)}_3 (c + 1) = 0;\\
	& a_2 \, c \, \mathcal{F}^{(2)}_3 (a_2 + 1)  - \left(a_2 + \phi\right) \, \left(c + \frac{1}{k} \Theta_{t} \right) \mathcal{F}^{(2)}_3 (c + 1) = 0;\\
	& a_2 \, \left(c + \frac{1}{k} \Theta_{t} - 1\right) \, \mathcal{F}^{(2)}_3 (a_2 + 1)  - (c - 1) \, \left(a_2 + \phi\right) \, \mathcal{F}^{(2)}_3 (c - 1) = 0;\\
	&	\left(a_2 + \phi - 1\right) \, \left(c + \frac{1}{k} \Theta_{t} - 1\right) \, \mathcal{F}^{(2)}_3 (a_2 - 1)  - (c - 1) \, (a_2 - 1)  \, \mathcal{F}^{(2)}_3 (c - 1) = 0;\\
	& c \,	\left(a_2 + \phi - 1\right) \,  \mathcal{F}^{(2)}_3 (a_2 - 1) -  (a_2 - 1)  \, \left(c + \frac{1}{k} \Theta_{t}\right) \mathcal{F}^{(2)}_3 (c + 1) = 0;\\
	& b_1 \,	\left(c + \frac{1}{k} \Theta_{t} -  1\right) \,  \mathcal{F}^{(2)}_3 (b_1 + 1) -  (c - 1)  \, \left(b_1 + \theta \right) \mathcal{F}^{(2)}_3 (c - 1) = 0;\\
	& b_1 \,	c \,  \mathcal{F}^{(2)}_3 (b_1 + 1) -  \left(c + \frac{1}{k} \Theta_{t}\right)  \, \left(b_1 + \theta\right) \mathcal{F}^{(2)}_3 (c + 1) = 0;\\
	& b_2 \,	\left(c + \frac{1}{k} \Theta_{t} - 1\right) \,  \mathcal{F}^{(2)}_3 (b_2 + 1) -  (c - 1)  \, \left(b_2 + \phi\right) \mathcal{F}^{(2)}_3 (c - 1) = 0;\\
	& b_2 \,	c \,  \mathcal{F}^{(2)}_3 (b_2 + 1) -  \left(c + \frac{1}{k} \Theta_{t}\right)  \, \left(b_2 + \phi\right) \mathcal{F}^{(2)}_3 (c + 1) = 0;\\
	& \left(b_1 + \theta - 1\right) \,	\left(c + \frac{1}{k} \Theta_{t} - 1\right) \,  \mathcal{F}^{(2)}_3 (b_1 - 1) -  (c - 1)  \, (b_1 - 1) \mathcal{F}^{(2)}_3 (c - 1) = 0;\\
	& 	c \, \left(b_1 + \theta - 1\right) \, \mathcal{F}^{(2)}_3 (b_1 - 1) - (b_1 - 1)  \left(c + \frac{1}{k} \Theta_{t}\right)  \,  \mathcal{F}^{(2)}_3 (c + 1) = 0;\\
	& \left(b_2 + \phi - 1\right) \,	\left(c + \frac{1}{k} \Theta_{t} - 1\right) \,  \mathcal{F}^{(2)}_3 (b_2 - 1) -  (c - 1)  \, (b_2 - 1) \mathcal{F}^{(2)}_3 (c - 1) = 0;\\
	& 	c \, \left(b_2 + \phi - 1\right) \, \mathcal{F}^{(2)}_3 (b_2 - 1) - (b_2 - 1)  \left(c + \frac{1}{k} \Theta_{t}\right)  \,  \mathcal{F}^{(2)}_3 (c + 1) = 0;\\
	& 	c \, (c - 1) \, \mathcal{F}^{(2)}_3 (c - 1) - \left(c + \frac{1}{k} \Theta_{t} - 1\right)  \left(c + \frac{1}{k} \Theta_{t}\right)  \,  \mathcal{F}^{(2)}_3 (c + 1) = 0.
\end{align}
\section{Conclusion}
The present paper is based on the study of two distinct discrete forms of Appell function $F_3$. We obtain their regions of convergence, difference-differential relations, integral representations, finite and infinite summation formulas, recursion relations etc. We also introduced the discrete analogues of Humbert functions $\varXi_1$ and $\varXi_2$ as limiting cases of discrete Appell functions $\mathcal{F}^{(1)}_3$ and $\mathcal{F}^{(2)}_3$.  
The particular cases of the results and identities obtained in this paper lead to the corresponding results for Appell function $F_3$ and Kamp\'e de F\'eriet  function. A few of them are already available in the literature but some are believed to be new.




\begin{thebibliography}{99}	
	

	\bibitem{bc1} M.\,Bohner, T.\,Cuchta, \textit{The Bessel difference equation.} Proc. Amer. Math. Soc. 145 (2017), no. 4, 1567--1580. 
	
	\bibitem{bc4} M.\,Bohner, T.\,Cuchta, \textit{The generalized hypergeometric difference equation.} Demonstr. Math., 51(2018), no. 1, 62--75.
	
	\bibitem{bc2} T.\,Cuchta, D.\,Grow, N.\,Wintz, \textit{Divergence criteria for matrix generalized hypergeometric series.} Proc. Amer. Math. Soc. 150 (2022), no. 3, 1235--1240. 
	
	\bibitem{bc3} T.\,Cuchta, R.\,Luketic, \textit{Discrete hypergeometric Legendre polynomials.} Mathematics, 9(20), 2021, 2546.
	
	\bibitem{bc5} T.\,Cuchta, D.\,Grow, N.\,Wintz, \textit{Discrete matrix hypergeometric functions.} J. Math. Anal. Appl. 518 (2023), no. 2, Paper No. 126716, 14~pp.
	
	\bibitem{kdf} P.\,Appell, J.\,Kamp\'e de F\'eriet, \textit{Fonctions hypergéometriques et hypersphériques. Polynômes d'Hermite}, Paris, Gauthier-Villars, 1926.
	
	\bibitem{ds1} R.\,Dwivedi, V.\,Sahai, \textit{On the hypergeometric matrix functions of two variables.} Linear Multilinear Algebra 66 (2018), no. 9, 1819--1837.
	
	\bibitem{ds2} R.\,Dwivedi, V.\,Sahai, \textit{On the basic hypergeometric matrix functions of two variables}.  Linear Multilinear Algebra 67 (2019), no. 1, 1--19. 
	
	\bibitem{ds5} R.\,Dwivedi, V.\,Sahai, \textit{A note on the Appell matrix functions}, Quaest. Math., (2019).  
		\bibitem{ds16} R.\,Dwivedi, V.\,Sahai, \textit{On the discrete analogues of Appell function $F_1$}, \url{https://arxiv.org/abs/2401.05677}.
			
		\bibitem{ds17} R.\,Dwivedi, V.\,Sahai, \textit{On the discrete analogues of Appell function $F_2$}, \url{https://arxiv.org/abs/2401.06501}.
	\bibitem{emo} A.\,Erd\'elyi, W.\,Magnus, F.\,Oberhettinger, F.\,G.\,Tricomi, \textit{Higher Transcendental Functions}, Vol. I, McGraw-Hill, New York, London, 1953.
	
	\bibitem{gr} G.\,Gasper, M.\,Rahman, \textit{Basic hypergeometric series},  Second edition. Encyclopedia of Mathematics and its Applications, 96. Cambridge University Press, Cambridge, 2004.
	
	\bibitem{ph} P.\,Humbert, \textit{The confluent hypergeometric functions of two variables.} Proc. Roy. Soc. Edinburgh, 41 (1920-21), 73--96.	
	%
	%
	%
	%
	%
	%
	%
	%
	%
	%
	%
	
	\bibitem{am} A.\,M.\,Mathai, \textit{ Appell's and Humbert's functions of matrix arguments.} Linear Algebra Appl. 183 (1993), 201--221.
	
	\bibitem{am1} A.\,M.\,Mathai, \textit{ Jacobians of matrix transformations and functions of matrix argument}, World Scientific Publishing Co., Inc., River Edge, NJ, 1997.
	
	\bibitem{mj} James\,A.\,Mullen, \textit{The differential recursion formulae for Appell's hypergeometric functions of two variables}. SIAM J. Appl. Math. 14 (1966), 1152--1163.
	
	\bibitem{nn} D.\,K.\,Nagar, S.\,Nadarajah, \textit{Appell's hypergeometric functions of matrix arguments}, Integral Transforms Spec. Funct. 28 (2017), no. 2, 91--112.
	
	\bibitem{edr} E.\,D.\,Rainville, \emph{Special Functions}, Chelsea, New York, 1960.
	
	%
	%
	%
	%

	\bibitem{sk} H.\,M.\,Srivastava, P.\,W.\,Karlsson, \textit{Multiple Gaussian Hypergeometric Series}, Ellis Horwood Limited, Chichester, 1985.
	
	\bibitem{bts} M.\,Tripathi, N.\,Saikia,  R.\,Barman, \textit{Appell's hypergeometric series over finite fields.} Int. J. Number Theory 16  (2020), no. 4, 673--692.
\end{thebibliography}
\end{document}